\documentclass{svjour3}                     
\smartqed  
\usepackage{amsmath,amsfonts,amssymb}
\usepackage{latexsym}
\usepackage{graphicx,overpic}
\usepackage{color}
\usepackage{wrapfig}
\usepackage{mathtools} 

\let\eps\varepsilon
\newcommand{\R}{\mathbb R}
\newcommand{\bA}{\mathbf A}
\newcommand{\cA}{\mathcal A}
\newcommand{\hcA}{{\widehat{\mathcal A}}}
\newcommand{\bB}{\mathbf B}
\newcommand{\cB}{\mathcal B}

\newcommand{\bC}{\mathbf C}

\newcommand{\bG}{\mathbf G}

\newcommand{\bg}{\mathbf g}

\newcommand{\bi}{\mathbf i}
\newcommand{\bj}{\mathbf j}
\newcommand{\bn}{\mathbf n}
\newcommand{\be}{\mathbf e}

\newcommand{\bu}{\mathbf u}

\newcommand{\bv}{\mathbf v}

\newcommand{\bx}{\mathbf x}

\newcommand{\bz}{\mathbf z}

\newcommand{\cD}{\mathcal D}

\newcommand{\cS}{\mathcal S}

\newcommand{\balpha}{\boldsymbol{\alpha}}
\newcommand{\hbalpha}{\widehat{\boldsymbol{\alpha}}}

\newcommand{\bbeta}{\boldsymbol{\beta}}

\newcommand{\halpha}{\widehat{\alpha}}
\newcommand{\bhalpha}{\widehat{\boldsymbol{\alpha}}}
\newcommand{\bchi}{{\boldsymbol{\chi}}}

\newcommand{\bseta}{\boldsymbol{\eta}}

\newcommand{\bPhi}{\boldsymbol{\Phi}}
\newcommand{\bPsi}{\boldsymbol{\Psi}}

\newcommand{\tr}{\widetilde{r}}

\newcommand{\tbPhi}{\widetilde{\bPhi}}

\newcommand{\rA}{\mathrm A}

\newcommand{\rC}{\mathrm C}

\newcommand{\rM}{\mathrm M}
\newcommand{\rP}{\mathrm P}

\newcommand{\rS}{\mathrm S}
\newcommand{\rU}{\mathrm U}
\newcommand{\rV}{\mathrm V}

\newcommand{\rF}{\mathrm F}

\newtheorem{algorithm}{Algorithm}
\newtheorem{prop}{Proposition}
\usepackage{todonotes}

\begin{document}

\title{Model order reduction of parametric dynamical systems by slice sampling tensor completion
}

\titlerunning{Model reduction of parametric systems by slice tensor completion}  

\author{Alexander V. Mamonov  \and Maxim~A.~Olshanskii}


\institute{A.V. Mamonov \at
Department of Mathematics, University of Houston, Houston, TX 77204\\
\email{avmamonov@uh.edu}
\and
M.A. Olshanskii \at
Department of Mathematics, University of Houston, Houston, TX 77204\\
\email{maolshanskiy@uh.edu}
}

\date{}

\maketitle

\begin{abstract}
Recent studies have demonstrated the great potential of reduced order modeling for parametric dynamical systems using low-rank tensor decompositions (LRTD). In particular, within the framework of interpolatory tensorial reduced order models (ROM), LRTD is computed for tensors composed of snapshots of the system's solutions, where each parameter corresponds to a distinct tensor mode. This approach requires full sampling of the parameter domain on a tensor product grid, which suffers from the curse of dimensionality, making it practical only for systems with a small number of parameters.  
To overcome this limitation, we propose a sparse sampling of the parameter domain, followed by a low-rank tensor completion. The resulting specialized tensor completion problem is formulated for a tensor of order $C + D$, where $C$ fully sampled modes correspond to the snapshot degrees of freedom, and $D$ partially sampled modes correspond to the system's parameters.  
To address this non-standard tensor completion problem, we introduce a low-rank tensor format called the hybrid tensor train. Completion in this format is then integrated into an interpolatory tensorial ROM. We demonstrate the effectiveness of both the completion method and the ROM on several examples of dynamical systems derived from finite element discretizations of parabolic partial differential equations with parameter-dependent coefficients or boundary conditions.

\keywords{Model order reduction \and 
Parametric dynamical systems \and Low-rank tensor completion \and Tensor decomposition}
\subclass{65P99 \and 15A69}
\end{abstract}

\section{Introduction}

In this paper, we focus on reduced order modeling for a multiparameter dynamical system with a sparse 
sampling of the parameter domain. 
{A particular setup that we study here is given by a dynamical system}
\begin{equation}
\label{eqn:GenericPDE}
\bu_t = F(t, \bu, \balpha),  \quad t \in (0,T), \quad \text{and}~ \bu|_{t=0} = \bu_0,
\end{equation}
where a known continuous flow field  $F:(0,T)\times \mathbb{R}^{M_1} \times\cA\to \mathbb{R}^{M_1}$
depends on the vector of parameters $\balpha= (\alpha_1,\dots,\alpha_D)^T$   from a parameter domain 
$\cA \subset \mathbb{R}^D$. 
A possible example of \eqref{eqn:GenericPDE} is a system of ODEs for nodal values of the finite 
volume or finite element solution to a parabolic PDE problem, where material coefficients, body forces, boundary conditions, the computational domain (via a mapping into a reference domain), etc., are parameterized by $\balpha$.

After discretization in time, each trajectory $\bu = \bu(t, \balpha) : [0,T) \to \mathbb{R}^{M_1}$ is represented by the collection of snapshots $\bu(t_k,\balpha)\in\mathbb{R}^{M_1}$ at times $0\le t_1 < \dots < t_{M_2}\le T$. The snapshots can be organized in a matrix 
\begin{equation}
\Phi(\balpha)=\big[\bu(t_1,\balpha),\dots, \bu(t_{M_2},\balpha)\big]\in \R^{M_1 \times M_2}.
\label{eqn:phialpha}
\end{equation}
Assuming that the parameter domain $\cA$ is the $D$-dimensional box 
\begin{equation}
\cA = {\textstyle \bigotimes\limits_{i=1}^D} [\alpha_i^{\min}, \alpha_i^{\max}],
\label{eqn:box}
\end{equation}  
we introduce a Cartesian grid on $\cA$ by distributing $K_i$ nodes $\{\halpha_i^j\}_{j=1,\dots,K_i}$  
within each of the intervals $[\alpha_i^{\min}, \alpha_i^{\max}]$  for $i=1,\dots,D$. 
The nodes of the grid form the set
\begin{equation}
\label{eqn:grid}
	\hcA = \left\{ \bhalpha =(\halpha_1,\dots,\halpha_D)^T:\,
	\halpha_i \in \{\halpha_i^k\}_{k=1}^{K_i}, ~ i = 1,\dots,D \right\},\quad  K= \prod_{i=1}^{D} K_i.
\end{equation} 

The \emph{discrete} parametric solution manifold is defined by the set of trajectories for all parameters 
$\bhalpha \in \hcA$, where each trajectory is given by the snapshot matrix $\Phi( \bhalpha)$. 
{Alternatively, this manifold can be represented} by the \emph{multi-dimensional} array
\begin{equation}
(\bPhi)_{:,:,k_1,\dots,k_D} = \Phi( \bhalpha),\quad 
\bhalpha=\big(\halpha_1^{k_1},\dots,\halpha_D^{k_D}\big)^T, 
\label{eqn:snapmulti}
\end{equation}
which is a tensor of order $2+D$ and size $M_1 \times M_2 \times K_1\times\dots\times K_D$. 
Observe that the first and second indices of the {snapshot tensor} $\bPhi$ are reserved for 
modes corresponding to the spatial and temporal degrees of freedom, respectively. Depending on any additional tensor structure exhibited by the state variable $\bu$, more than two modes of $\bPhi$ may be allocated for space-time indexing. In general, we assume $C \geq 1$ space-time modes in $\bPhi \in \mathbb{R}^{M_1 \times \dots \times M_C \times K_1 \times \dots \times K_D}$.

It is clear that for a high dimension $D$ or a fine parameter grid $\hcA$, working directly with $\bPhi$ 
is prohibitively expensive, if possible at all. Therefore, one approach to model order reduction for 
a system like \eqref{eqn:GenericPDE} is based on the assumption that $\bPhi$ can be effectively 
{approximated by a \emph{low-rank} tensor $\widetilde{\bPhi}$ in one of the 
commonly used formats. Depending on the definition of tensor rank, various tensor decomposition formats can be employed, including CANDECOMP/PARAFAC (CP), Tucker (HOSVD), or Tensor Train (TT).} 
While known rigorous analyses revealing how tensor ranks depend on the properties 
of a parameterized differential equation and the targeted accuracy of the solution manifold recovery, 
or its smoothness, are very limited (see, e.g., \cite{bachmayr2017kolmogorov} and discussions 
in \cite{khoromskij2011tensor,nouy2017low,bachmayr2023low}), numerical evidence suggests that 
for many practical parametric PDEs, solutions to such problems or statistics derived from them are well 
approximated in low-rank formats \cite{kressner2011low,ballani2015hierarchical,dolgov2015polynomial,ballani2016reduced,eigel2017adaptive,bachmayr2018parametric,dolgov2019hybrid,glau2020low,mamonov2022interpolatory,mamonov2024tensorial}.

Since entry-wise assembly of the snapshot tensor $\bPhi$ may not be feasible for large $D$,
a more practical approach for finding $\widetilde{\bPhi}$ is by solving a tensor completion problem.
A general low-rank tensor completion problem can be formulated as finding a minimal rank tensor 
{$\tbPhi\in\R^{M_1 \times \dots \times M_C \times K_1 \times \dots \times K_D}$ 
that fits the tensor 
$\bPhi\in\R^{M_1 \times \dots \times M_C \times K_1 \times \dots \times K_D}$ 
for a subset of its observed entries: 
\begin{equation}
\label{CProblem}
\tbPhi = \underset{\bPsi\in\R^{M_1 \times \dots \times M_C \times K_1 \times \dots \times K_D}}{\operatorname{argmin}}\mbox{rank}(\bPsi),
\quad \text{s.t.} \quad \bPsi|_\Omega=\bPhi|_\Omega,
\end{equation}}
where $\bPsi|_\Omega$ is a restriction of the tensor on the set of indices $\Omega$ of observed entries. 
The exact fitting in \eqref{CProblem} can be relaxed to the approximate one, yielding the inexact completion 
problem
\begin{equation}
\label{CProblem2}
\tbPhi = \underset{\bPsi\in\R^{M_1 \times \dots \times M_C \times K_1 \times \dots \times K_D}}{\operatorname{argmin}}\mbox{rank}(\bPsi),
\quad \text{s.t.}~
\|\bPsi|_\Omega-\bPhi|_\Omega\|\le \varepsilon,
\end{equation}
with a prescribed $\varepsilon \ge 0$. The inexact completion \eqref{CProblem2} is a commonly used 
problem setup for the case of noisy data~\cite{candes2010matrix} and this is the formulation we are 
interested in here. 
For CP, Tucker, and TT tensor ranks, the completion problems \eqref{CProblem} and \eqref{CProblem2} 
are NP-hard. {Building on the success of solving low-rank matrix completion problems
\cite{candes2010power,cai2010singular,candes2012exact,vandereycken2013low,jain2013low,nguyen2019low},
a popular approach is to relax \eqref{CProblem}--\eqref{CProblem2} into convex optimization problems
\cite{signoretto2010nuclear,gandy2011tensor,bengua2017efficient}.}
Other approaches to  {approximate} tensor completion include ALS methods
\cite{yokota2016smooth,grasedyck2019stable}, Riemannian optimization
\cite{kressner2014low,steinlechner2016riemannian}, Bayesian methods
\cite{zhao2015bayesian,zhao2015bayesian2}, and projection methods~\cite{rauhut2015tensor}, 
with applications ranging from video recovery to seismic data reconstruction
\cite{bengua2017efficient,liu2022efficient}.


While the completion problem formulations \eqref{CProblem}--\eqref{CProblem2} accommodate a wide range of applications, model order reduction for systems like \eqref{eqn:GenericPDE} imposes additional structure on the set $\Omega$. In particular, we focus on the case where $\Omega$ corresponds to the so-called {slice sampling}.  
As indicated by the notation above, in slice sampling, the set of all $C + D$ tensor modes is divided into two groups. The first $C$ modes, $M_1 \times \dots \times M_C$, are fully sampled, while the remaining $D$ modes, $K_1 \times \dots \times K_D$, are sampled sparsely. For a fixed set of $D$ indices $k_1, \dots, k_D$ corresponding to sparsely sampled modes, we define the tensor  
\begin{equation}
(\bPhi)_{:,\ldots,:,k_1,\dots,k_D} \in \mathbb{R}^{M_1 \times \dots \times M_C}
\label{eqn:phislice}
\end{equation}  
as a \textit{slice}, hence the name \textit{slice sampling}. For the problem \eqref{eqn:GenericPDE}, we have $C = 2$,  {and in general we assume that $C$ is small so that using a Cartesian grid in physical variables is feasible.} The slice \eqref{eqn:phislice} coincides with the snapshot matrix \eqref{eqn:phialpha}, where $\balpha = \big(\halpha_1^{k_1}, \dots, \halpha_D^{k_D} \big)^T$. The sampled indices $k_1, \dots, k_D$ form a sparse subset of the grid $\hcA$.  

To exploit slice sampling efficiently and improve the computational cost of solving the completion problem \eqref{CProblem2}, we introduce a custom low-rank tensor format called the \textit{hybrid tensor train} (HTT). Completion in the HTT format involves projecting the slice-sampled tensor onto reduced orthogonal bases along the first $C$ modes (similar to HOSVD), followed by multiple completions of smaller tensors in TT format. These smaller completions can be performed in parallel using existing tensor completion algorithms.  
The HTT format is particularly well-suited for constructing a Galerkin reduced-order model (ROM) known in the literature as the {tensorial ROM (TROM)}~\cite{mamonov2022interpolatory,mamonov2023analysis,mamonov2024tensorial}, a natural extension of the POD ROM for parametric dynamical systems.

The application of tensor completion methods to reduced order modeling of parametric ODEs or PDEs 
is rare, and this study aims to explore this direction. In general, the use of tensor methods for solving 
parametric PDEs is not new. Several studies have developed sparse tensorized solvers for certain 
high-dimensional and stochastic PDEs
~\cite{schwab2011sparse,khoromskij2011tensor,dolgov2015polynomial,garreis2017constrained,nouy2017low,dolgov2018direct,dolgov2019hybrid}.

The reconstruction of scalar output quantities of parametric solutions in tensor format from incomplete 
observations was addressed in~\cite{ballani2015hierarchical,glau2020low}. In~\cite{ballani2015hierarchical}
the authors employed a tensor cross approximation, while~\cite{glau2020low} applied TT-completion via 
Riemannian optimization to recover an option pricing statistic from solutions of parametrized Heston and 
multi-dimensional Black-Scholes models. Additionally, a comparison of TT-cross interpolation and 
TT-completion for a parameterized diffusion equation in~\cite{steinlechner2016riemannian} demonstrated 
that TT-completion requires fewer PDE solver executions to find a low-rank approximation of a particular 
solution functional.

While the works~\cite{ballani2015hierarchical,steinlechner2016riemannian,glau2020low} focused on 
recovering scalar solution statistics in tensor format, here we aim to approximate a tensor of solution 
snapshots \label{eqn:snapmulti} for subsequent use in building TROM. Similarly to the previous studies
~\cite{mamonov2022interpolatory,mamonov2023analysis,mamonov2024tensorial}, 
we seek a low-rank approximation of a tensor. {However, here we compute such an 
approximation from a sparse sampling of the tensor.} Due to the relatively high separation ranks and 
large sizes of the space and time modes in the snapshot tensor, applying existing completion 
algorithms in standard low-rank formats is computationally prohibitive. 
{This motivates the introduction of tensor completion in the HTT format customized 
specifically for slice-sampled tensors arising in TROM construction.}

The remainder of the paper is organized as follows: Section~\ref{sec2} provides a more detailed 
explanation of slice sampling and introduces the proposed completion method. We deviate from the traditional 
approach of defining the tensor rank before formulating the completion problem, as we find it more 
instructive to first explain the method of obtaining the fitting tensor. The resulting rank-revealing format 
becomes more intuitive afterward. Section~\ref{sec3} describes the Galerkin ROM for the dynamical 
system. This ROM utilizes HTT as the dimension reduction technique (in place of the standard POD), 
and we refer to it as HTT-ROM. Section~\ref{sec4} presents the results of numerical experiments.

\section{Slice sampling tensor completion} 
\label{sec2}

We consider here a problem that we refer to hereafter as slice sampling low-rank tensor completion. 
Consider a tensor $\bPhi$ of order $C + D$ of size 
$M_1 \times \ldots \times M_{C} \times K_1 \times \ldots \times K_{D}$, and two sets of indices
\begin{equation}
\Omega_{C} = \bigotimes\limits_{i=1}^{C} \{ 1, \ldots, M_i \}, \quad 
\Omega_D = \bigotimes\limits_{j=1}^{D} \{ 1, \ldots, K_j \},
\label{eqn:omegafp}
\end{equation}
where the products are understood in Cartesian sense. The set $\Omega_{C}$ is the set of \emph{sliced} 
indices and  $\Omega_{D}$ is the set of \emph{sampled} indices.
The sampling (or training) set is the subset 
\begin{equation}
\widetilde\Omega_D \subset \Omega_D. 
\label{eqn:omegap}
\end{equation}

In what follows we employ multi-indices
\begin{equation}
\bi = (i_1,\ldots,i_{C}), \quad
\bj = (j_1,\ldots,j_{D}),
\end{equation}
for the first $C$ and last $D$ indices of $\bPhi$, respectively.
Then, given the data
\begin{equation}
\cD = \left\{ 
\Phi_{\bi, \bj} = 
\Phi_{i_1,\ldots,i_{C},j_1,\ldots,j_{D} } \;
\left| \; \forall\, \bi \in \Omega_{C}, \;
\forall\, \bj \in \widetilde{\Omega}_D
\right.\right\},
\label{eqn:datafp}
\end{equation}
we seek 
a completing tensor $\widetilde{\bPhi}$ solving \eqref{CProblem} or \eqref{CProblem2} with $\Omega=\Omega_{C}\otimes\widetilde{\Omega}_D$.
\smallskip

Examining the definition of the index sets \eqref{eqn:omegafp}--\eqref{eqn:omegap}, we observe that the term \emph{slice sampling} refers to the setup in which the first $C$ modes of $ \bPhi $, the space--time modes of the snapshot tensor, are fully sampled, while the last $D$ modes are only partially sampled. In other words, we learn the tensor $ \bPhi $ through $C$-dimensional slices. This distinguishes the slice sampling completion problem considered here from conventional low-rank tensor completion settings, in which none of the tensor modes are fully sampled.

Given that the two index sets \eqref{eqn:omegafp} of entries of $ \bPhi $ are sampled differently, slice sampling tensor completion is a two-stage process. First, orthonormal bases for the low-dimensional subspaces of the fully sampled modes are computed. Second, TT-completion is performed component-wise in the subspaces computed in the first stage. Finally, the completed tensor $ \widetilde{\bPhi} $ can be assembled, or one may use the reduced bases and component-wise TT-completions from the first two stages to operate on $ \widetilde{\bPhi} $ without assembling it explicitly.

\subsection{Reduced basis for fully sampled space--time modes}

As mentioned above, the first stage of slice sampling completion is the computation of the 
orthonormal bases for the reduced  subspaces corresponding to space--time modes of $\bPhi$.
Since the first $C$ modes are fully sampled, one may employ a simple technique based on truncated SVD to compute the bases of interest. 

First, for all multi-indices in the set $\widetilde{\Omega}_D$ introduce a linear enumeration so that it can be written as
\begin{equation}
\widetilde{\Omega}_D = \left\{ \left. \bj^{(k)} \in \Omega_D \; \right| \; k = 1,\ldots,P \right\},\quad P=|\widetilde{\Omega}_D|.
\label{eq:OmegaP}
\end{equation}
Consider a tensor $\bPhi^P$ of order $C + 1$ of size $M_1 \times \ldots \times M_{C} \times P$
with entries 
\begin{equation}\label{eq:PhiP}
\Phi^P_{\bi,k} = \Phi_{\bi,\bj^{(k)}}, \quad
k = 1,\ldots,P.
\end{equation}
Then, assemble $C$ matrices $\rF^{(i)} \in \mathbb{R}^{M_i \times {M}_i^\prime P}$, $i=1,\ldots,C$, 
with $M_i^\prime = \prod\limits_{{q = 1}\atop{q \neq i}}^{C} M_q$, defined as
\begin{equation}\label{eq:Fi}
\rF^{(i)} = \mbox{unfold}_i \Big( \bPhi^P \Big),
\end{equation}
where $\mbox{unfold}_i$ denotes the $i^{\text{th}}$-mode unfolding of a tensor. 
Next, compute the singular value decompositions
\begin{equation}\label{eq:SVD1}
\rF^{(i)} = \rU^{(i)} \Sigma^{(i)} (\rV^{(i)})^T, \quad i = 1,\ldots,C,
\end{equation}
where each $\Sigma^{(i)}$ contains the singular values $\sigma^{(i)}_1 \geq \ldots \geq \sigma^{(i)}_{M_i} $.
Choose a threshold $\eps_{C}\ge0$ and determine the ranks
\begin{equation}\label{eq:qi}
q_i = \min \Big\{ q ~\Big|~ \sum_{j>q}\big(\sigma^{(i)}_j\big)^2  \leq \eps_{C} \big\|\rF^{(i)}\big\|_{F}^2 \Big\}, \quad i = 1,\ldots,C.
\end{equation}
Using Matlab notation, we denote the matrix containing the first $q_i$ columns of $\rU^{(i)}$ by
\begin{equation}
\widetilde{\rU}^{(i)} = \rU^{(i)}_{:, 1:q_i} 
\in \mathbb{R}^{M_i \times q_i}, 
\quad i = 1,\ldots,C.
\label{eqn:tui}
\end{equation}
The columns of matrices $\widetilde{\rU}^{(i)}$ comprise the orthonormal bases that we refer to as the reduced bases for all fully sampled modes, $i = 1,\ldots,C$.

\subsection{Tensor-train completion for partially sampled modes} 
\label{s:compl2}

Once the reduced bases for the fully sampled space--time modes are computed, there exist two options for completion along the partially sampled ones. In order to describe both we recall the definition of $n$-mode tensor-matrix product. Given a tensor $\bPsi$ of order $d$ of size 
$m_1 \times \ldots \times m_d$ and a matrix $\rM \in \mathbb{R}^{m \times m_n}$, the $n$-mode product $\bPsi \times_n \rM$
is a tensor of order $d$ of size 
$m_1 \times m_{n-1} \times m \times m_{n+1} \times m_d$ with entries given by
\begin{equation}
[\bPsi \times_n \rM]_{i_1,\ldots,i_{n-1},j,i_{n+1},\ldots,i_d} = \sum_{k = 1}^{m_n}
\bPsi_{i_1,\ldots,i_{n-1},k,i_{n+1},\ldots,i_d} 
\rM_{jk}, \quad j = 1,\ldots,m.
\end{equation}

Consider the projected tensor 
\begin{equation}
\bPhi^q = \bPhi \times_1 
\left( \widetilde{\rU}^{(1)} \right)^T \times_2
\left( \widetilde{\rU}^{(2)} \right)^T \times_3 
\ldots \times_{C} 
\left( \widetilde{\rU}^{(C)} \right)^T
\end{equation}
of order $C + D$ and of size 
$q_1 \times \ldots \times q_{C} \times K_1 \times \ldots \times K_{D}$. With the reduced bases of fully sampled modes at hand, one may compute from the original data $\cD$ the projected data
\begin{equation}
\cD^q = \left\{ \bPhi^q_{\bi,\bj} \;
\left| \; \forall \bi \in \Omega_{C}^q, \;
\forall \bj \in \widetilde{\Omega}_D
\right.\right\},
\label{eqn:dataq}
\end{equation}
where
\begin{equation}
\Omega_{C}^q = \bigotimes\limits_{i=1}^{C} \{ 1, \ldots, q_i \}.
\label{eqn:omegafr}
\end{equation}
Then, the key step in finding $\widetilde{\bPhi}$ from $\cD$ is to determine $\widetilde{\bPhi}^q$, a low-rank completion of $\bPhi^q$ from the projected data $\cD^q$. 
Once $\widetilde{\bPhi}^q$ is found, the completion $\widetilde\bPhi$  of $\bPhi$ is given by
\begin{equation}\label{eq:tPhi}
\widetilde\bPhi = \widetilde{\bPhi}^q \times_1 
\widetilde{\rU}^{(1)} \times_2
\widetilde{\rU}^{(2)}  \times_3 
\ldots \times_{C} 
\widetilde{\rU}^{(C)}.
\end{equation}
One can easily check\footnote{To verify the result in \eqref{eq:estEps}, we note that $\|\bPhi|_\Omega-\tbPhi|_\Omega\|_F=
\|\bPhi^P-\tbPhi^P\|_F$, $\|\bPhi|_\Omega\|_F=\|\bPhi^P\|_F$ and apply \cite[Property 10]{de2000multilinear} after observing that $\tbPhi^P$ is a truncated HOSVD of $\bPhi^P$.} 
that $\tbPhi$ satisfies 
\begin{equation}\label{eq:estEps}
\big\|\bPhi|_\Omega-\tbPhi|_\Omega\big\|_F\le \sqrt{C}\eps_C\big\|\bPhi|_\Omega\big\|_F,
\end{equation}
with $\eps_C$ from \eqref{eq:qi}. For $\eps_C>0$ this renders our completion inexact as in \eqref{CProblem2}.  

We still need to address the completion problem for a sliced-sampled tensor, but with reduced dimensions of the slices.  
As mentioned previously, there are two ways of solving this problem that we consider below.

The first option is to find the whole tensor $\widetilde{\bPhi}^q$ in a common low-rank tensor format such as TT. Further in the paper we work with with low-rank completion in TT format, but it is certainly possible to use other LRTDs, e.g., in CP or Tucker formats, as well. 
Finding $\widetilde{\bPhi}^q$  can be achieved using an existing method for low-rank tensor completion, e.g., stable ALS method from \cite{grasedyck2019stable}. 
The main disadvantage of such approach applied to, e.g., model order reduction for parametric dynamical systems, is its high cost both in terms of memory and computational time. To resolve it, we suggest a different approach introduced as the second option below.

The second option is to perform a low-rank TT completion component-wise in the following sense. Let
$Q = \left| \Omega_{C}^q \right| = \prod_{i=1}^{C} q_i$. Consider $Q$
data sets
\begin{equation}
\cD^q_{\bi} = \left\{ \Phi^q_{\bi,\bj} \;
\left| \; \forall \bj \in \widetilde{\Omega}_D
\right.\right\}, \quad 
\bi \in \Omega_{C}^q.
\label{eqn:dataqi}
\end{equation}
Then, we perform $Q$ low-rank TT completions with data 
$\cD^q_{\bi}$ for each $\bi \in \Omega_{C}^q$ to obtain TT-tensors 
\begin{equation}
\widetilde{\bPhi}^q_{\bi} = 
\sum_{j_0=1}^{r_0^{\bi}} \cdots \sum_{j_{D}=1}^{r_{D}^{\bi}}  \bg^{\bi}_{1,j_0,j_1} \circ \cdots \circ  \bg^{\bi}_{D,j_{D-1},j_D},\quad \bi \in \Omega_{C}^q,
\label{eqn:bphiqitt}
\end{equation}
where $\bg^{\bi}_{k,j_{k-1},j_k} \in \R^{N_k}$ for $k=1,\ldots,D$, and for all values of $j_0,\ldots,j_{D}$ that appear in \eqref{eqn:bphiqitt}. The summation indices  $r_k^{\bi}$ are the TT compression ranks, where we follow the convention $r_0^{\bi} = r_D^{\bi} = 1$ for notation convenience. 

After computing the completions \eqref{eqn:bphiqitt},
the entries of the tensor $\widetilde{\bPhi}^q$ are simply
\begin{equation}
\widetilde{\bPhi}^q_{\bi, \bj } = \left[\widetilde{\bPhi}^q_{\bi}\right]_{\bj}, 
\quad \bi \in \Omega_{C}^q, 
\quad \bj \in \Omega_D.
\end{equation}
Then, the completion $\widetilde\bPhi$ of $\bPhi$ can be computed via \eqref{eq:tPhi}.

Summarizing, the HTT format we introduced to complete a sliced-sampled tensor is given by 
$C$ matrices and $Q$ tensors in TT format:
\begin{equation}\label{eq:HTT}
\rm{HTT}\big( \widetilde\bPhi \big) = \left\{\widetilde{\rU}^{(i)} \in \mathbb{R}^{M_i \times q_i}, \quad i = 1,\ldots,C, 
\quad \widetilde{\bPhi}^q_{\bi}\in \R^{K_1\times\dots\times K_D},~\bi\in \Omega_{C}^q
\right\}.
\end{equation}
We call $\{q_i\}_{i=1,\dots,C}$ the \emph{C-ranks} and $\{r^{\bi}_i\}_{i=0,\dots,D}$ the \emph{D-ranks} of $\widetilde\bPhi$.

Assuming for simplicity that $M_i=M$, $K_j=K$, all C-ranks are equal to $q$ and all D-ranks are equal to $r$, the representation complexity of $\bPhi$ is $CMq + (q^C)D(r^2)K$. We observe that the complexity grows exponentially in $C$. However,
in applications we are interested in, $C$ is typically small. 
In Section~\ref{sec3}, we provide the details of the case when the slice sampling completion is applied to build a ROM for a general parametric dynamical system with $C=2$. 

\subsection{Algorithm for slice sampling tensor completion} 
\label{s:compl3}

We summarize the discussion of slice sampling tensor completion in HTT format in the algorithm below.

\begin{algorithm}[Slice sampling tensor completion in HTT format]
\label{Alg1}
\label{alg:stcomptt}
~\\[1ex]
\textbf{Input:} Index set $\widetilde{\Omega}_D$ as in \eqref{eqn:omegap} and the corresponding data $\cD$ as in \eqref{eqn:datafp}, threshold $\eps_{C}\ge0$.
\begin{enumerate}
\item Form the tensor $\bPhi^P$ from the data $\cD$ as in \eqref{eq:PhiP}. 
\item For $i = 1,\ldots,C$:
\begin{itemize}
\item[(a)] Form the unfolding $\rF^{(i)} = \mbox{unfold}_i \big( \bPhi^P \big)$;
\item[(b)] Compute the SVD of the unfolding matrix $\rF^{(i)} = \rU^{(i)} \Sigma^{(i)} (\rV^{(i)})^T$;
\item[(c)] Choose the rank $q_i$ using \eqref{eq:qi} and form the matrix $\widetilde{\rU}^{(i)} = \rU^{(i)}_{:, 1:q_i}$, the columns of which form the reduced basis. 
\end{itemize}

\item For each multi-index $\bi \in \Omega^q_{C}$:
\begin{itemize}
\item[(a)] Compute the projected data 
\begin{equation}
\cD^q_{\bi} = \left\{ 
\Phi_{\bi,\bj^{(k)}} \times_1 
\big( \widetilde{\rU}^{(1)} \big)^T \times_2
\big( \widetilde{\rU}^{(2)} \big)^T \times_3 
\ldots \times_{C} 
\big( \widetilde{\rU}^{(C)} \big)^T
\right\}_{k=1}^{P},
\end{equation}
where the multi-indices $\bj^{(k)}$ are as in \eqref{eq:OmegaP};
\item[(b)] Perform low-rank completion in TT format \eqref{eqn:bphiqitt} with data $\cD^q_{\bi}$ to obtain TT tensor $\widetilde{\bPhi}^q_{\bi}$ with compression ranks $r^{\bi}_k$, $k=0,1,\ldots,D$.
\end{itemize}
\end{enumerate}
\noindent \textbf{Output:} 
$\rm{HTT} \big( \tbPhi \big)$ as in \eqref{eq:HTT}.
\end{algorithm}

\subsection{An estimate for the $\{C,D\}$-ranks} 

Let us prove one  estimate of the $C$- and $D$-ranks of a tensor through the  ranks of its unfoldings. To this end, denote by $\mbox{matr}_k(\bA)\in\R^{(K_1\dots K_k)\times(K_{k+1}\dots K_N)}$ the $k$-index matricization of a tensor $\bA\in \R^{K_1\times\dots\times K_N}$ (see, e.g., \cite{hackbusch2012tensor} for a definition) 
and let 
\begin{equation}
\widehat q_i=\mbox{rank}(\mbox{unfold}_i(\bPhi)), \quad\text{and}\quad \widehat r_k=\mbox{rank}(\mbox{matr}_k(\bPhi)),
\end{equation}
for $i=1,\dots,\text{\small $C+D$}$, $k=1,\dots,\text{\small $C+D$}-1$, and $\widehat r_0=\widehat r_{C+D}=1$. 
In other words,  $\{\widehat q_1,\dots,\widehat q_{C+D}\}$ and $\{\widehat r_0, \dots,\widehat r_{C+D}\}$ are the HOSVD and TT-SVD ranks of $\bPhi$, respectively.

For the purpose of analysis, it is convenient to think about TT decomposition as a tensor chain. For example, for the decomposition \eqref{eqn:bphiqitt} we 
define $r_{k-1}^{\bi}\times N_k\times r_k^{\bi} $ tensors $\bG_{\bi}^{(k)}$ by $\bG_{\bi}^{(k)}=\bg^{\bi}_{k,:,:}$ and consider 
the representation of $\widetilde{\bPhi}^q_{\bi}$ via the chain of tensors,
\begin{equation}
\widetilde{\bPhi}^q_{\bi} \sim \left\{ \bG_{\bi}^{(1)},\dots, \bG_{\bi}^{(D)}\right\}.
\end{equation}

\begin{prop}
\label{prop:httranks}
 Tensor $\bPhi$ can be represented in the HTT format with $C$-ranks $\{q_1,\dots,q_{C}\}$  and $D$-ranks
    $\{r_0^{(\bi)},\dots,r_{D}^{(\bi)}\}$, ${\bi\in \Omega_{C}^q}$, 
    so that it holds
    \begin{equation}
    q_i = \widehat q_i,\quad \text{and} \quad
     r_0^{(\bi)}=1,~ r_j^{(\bi)}= \widehat r_{j+C},
    \end{equation}
    {for} $i=1,\dots,\text{\small $C$}$, $j=1,\dots,\text{\small $D$}$, and for all $\bi \in \Omega_{C}^q$. 
\end{prop}
\begin{proof}  Let $\{\bC,\rU^{(1)},\dots,\rU^{(C+D)}\}$ be the HOSVD decomposition of $\bPhi$ with the core tensor $\bC$. We note the  identity 
\begin{equation}
\mbox{matr}_k(\bPhi) = \left(\rU^{(k)}\otimes\dots\otimes\rU^{(1)}\right)\mbox{matr}_k(\bC) 
\left(\rU^{(C+D)}\otimes\dots\otimes\rU^{(k+1)}\right)^T,
\end{equation}
where the Kronecker tensor products of $\rU^{(i)}$ are the matrices with orthonormal columns. This implies $\mbox{rank}(\mbox{matr}_k(\bPhi))=\mbox{rank}(\mbox{matr}_k(\bC))$ and hence
the TT-SVD ranks of $\bC$ and $\bPhi$ are the same.
Further, let 
\begin{equation}
\bC\sim\left\{\bG^{(1)},\cdots, \bG^{(C+D)} \right\}
\end{equation}
be the TT-SVD decomposition of $\bC$, where 
$\bG^{(j)}$ are order three tensors of sizes $\widehat r_{j-1}\times \widehat q_j\times  \widehat r_{j}$, and define $\widehat \bG^{(j)}= \bG^{(j)}\times_2 \rU^{(j)}$ for $j=\text{\small $C$}+1,\dots,\text{\small $C+D$}$.  
Then, an HTT decomposition of $\bPhi$ is given by the matrices $\{\rU^{(1)},\dots,\rU^{(C)}\}$ and the TT tensors
\begin{equation}\label{aux434}
   \widetilde\bPhi^q_\bi\sim \left\{\widetilde \bG^{(C+1)}_\bi,\widehat \bG^{(C+2)},\cdots,\widehat \bG^{(C+D)}\right\}, 
\end{equation}
where $\widetilde \bG^{(C+1)}_\bi= \widehat \bG^{(C+1)}\times_1 \bv_\bi$ and $\bv_\bi\in\R^{\widehat r_C}$ is a vector defined through the product of matrices:  $\bv_\bi=\bG^{(1)}_{:,i_1,:}\cdots \bG^{(C)}_{:,i_{C},:}$.  
\end{proof}

The HTT representation guaranteed by Proposition~\ref{prop:httranks} does not necessarily have minimal HTT ranks. In particular, in the representation \eqref{aux434}, only $ \widetilde{\bG}^{(C+1)} $ depends on $ \bi $, and for all $ \bi $, the constructed tensors $ \widetilde{\bPhi}^q_\bi $ have the TT representation with the same ranks, which can be non-optimal. Moreover, we note that the HTT decomposition, by its construction, is independent of the position of the sliced indices in the tensor $ \bPhi $. We consider the first $ C $ for convenience. However, the TT ranks of $ \bPhi $ do depend on the permutation of indices. Therefore, the result of Proposition~\ref{prop:httranks} can be improved to achieve the minimum of TT ranks over all such permutations of modes in $\bPhi$ that preserve the order of the $D$-modes.
 
\subsection{Accuracy of HTT completion} 

In this section, we prove an estimate for the accuracy of HTT completion in terms of the accuracy of the three steps for computing $\tbPhi$: sampling, projection, and component-wise completion.  

We recall that $\rU^{(i)}$, $i=1,\dots,C$, are orthogonal matrices of left singular vectors of $i$-th mode unfoldings of the tensor $\bPhi^P$, cf. \eqref{eq:PhiP} and \eqref{eq:Fi}.   Denote the orthogonal projection on the column space of  $\rU^{(i)}$ by $\rP_{(i)}=\rU^{(i)}\big(\rU^{(i)}\big)^T$ and define
\begin{equation}\label{eq:deltai}
    \eps_s=\max\limits_{i=1,\ldots,C}\frac{\|\rP_{(i)}^\perp \Phi^{(i)}\|_{F}}{\|\Phi^{(i)}\|_{F}}, \quad \text{with}~
    \Phi^{(i)}=\mbox{unfold}_i(\bPhi),
\end{equation}
where $\rP_{(i)}^\perp = \mathrm{I} - \rP_{(i)}$.
The values of $\eps_s$ quantify how representative is the sampling of $\bPhi$, with smaller values corresponding to better representation.

Next, the accuracy of the  projected tensor $\bPhi^q$ depends on the threshold $\eps_{C}\ge0$ used in \eqref{eq:qi} to determine the $q$-ranks. Finally, we assume that during the component-wise completion the tensors $\bPhi_\bi^q:=\bPhi_{\bi,:}^q$ are reconstructed with the accuracy $\eps_q\ge0$, i.e., the bound
\begin{equation}\label{eq:epsq}
\|\bPhi_\bi^q-\tbPhi_\bi^q\|_{F}\le \eps_q\|\bPhi_\bi^q\|_{F}
\end{equation}
holds uniformly for all $\bi\in\Omega_{C}^q$.
\smallskip

The following lemma provides the estimate of the completion accuracy.
\begin{prop}\label{L2}
    The following estimate holds
    \begin{equation}\label{eq:accuracy}
        \|\bPhi-\tbPhi\|_{F}\le \big(\sqrt{C}(\eps_{C}+\eps_s)+\eps_q)\big)\|\bPhi\|_{F},
    \end{equation}
with $\eps_{C}$, $\eps_s$, and $\eps_q$ from \eqref{eq:qi}, \eqref{eq:deltai},  and \eqref{eq:epsq}, respectively.
\end{prop}
\begin{proof} By the triangle inequality we have
\begin{equation}\label{aux473}
      \|\bPhi-\tbPhi\|_{F}\le  \|\bPhi-\bPhi^q \times_1 \widetilde{\rU}^{(1)} 
\ldots \times_{C} \widetilde{\rU}^{(C)}\|_{F}+\|(\bPhi^q-\widetilde{\bPhi}^q) \times_1 
\widetilde{\rU}^{(1)}  \ldots \times_{C} \widetilde{\rU}^{(C)}\|_{F}.
\end{equation}
Using the definition of the Frobenious norm, its invariance under the tensor-matrix products with orthogonal $\widetilde{\rU}^{(i)}$'s, and \eqref{eq:epsq}, we estimate  the second term on the right hand side:
\begin{equation}
\label{aux496}
\begin{aligned}
\|(\bPhi^q-\widetilde{\bPhi}^q) \times_1 
\widetilde{\rU}^{(1)}  \ldots \times_{C} \widetilde{\rU}^{(C)}\|_{F}^2 & =
\|\bPhi^q-\widetilde{\bPhi}^q\|_{F}^2=
\sum_{\bi \in \Omega_C^q}\|\bPhi^q_\bi-\widetilde{\bPhi}^q_\bi\|_{F}^2 \\
& \le  \eps_q^2\sum_{\bi \in \Omega_C^q} \|\bPhi_\bi^q\|_{F}^2=\eps_q^2\|\bPhi\|_{F}^2.
 \end{aligned}
\end{equation}
To handle the first term on the right hand side of  \eqref{aux473}, define auxiliary tensor $\widehat\bPhi=\bPhi \times_1 {\rP}_{(1)} 
\ldots \times_{C} \rP_{(C)}$. It is easy to see that $\bPhi^q \times_1 \widetilde{\rU}^{(1)} 
\ldots \times_{C} \widetilde{\rU}^{(C)}$ is the truncated HOSVD of  $\widehat\bPhi$ and so it satisfies (see \cite[Property 10]{de2000multilinear}) the bound 
\begin{equation} 
\|\widehat\bPhi-\bPhi^q \times_1 \widetilde{\rU}^{(1)} 
\ldots \times_{C} \widetilde{\rU}^{(C)}\|_{F}\le\sqrt{C}\eps_{C}\|\widehat\bPhi\|_{F}.
\end{equation}
Since $\widehat\bPhi$ is an orthogonal projection of $\bPhi$ along the first $C$ modes, it also holds $\|\widehat\bPhi\|_{F}\le\|\bPhi\|_{F}$.
Hence by the triangle inequality we obtain
\begin{equation}
\label{aux510}
\|\bPhi-\bPhi^q \times_1 \widetilde{\rU}^{(1)} 
\ldots \times_{C} \widetilde{\rU}^{(C)}\|_{F}\le 
\|\bPhi-\widehat\bPhi\|_{F}+ 
\sqrt{C}\eps_{C}\|\bPhi\|_{F}.
\end{equation}
Consider now the decomposition 
\begin{equation}
\begin{aligned}
\bPhi-\widehat\bPhi=\bPhi \times_1 {\rP}_{(1)}^\perp + 
(\bPhi \times_2 {\rP}_{(2)}^\perp)\times_1 {\rP}_{(1)} + (\bPhi \times_3 {\rP}_{(3)}^\perp)\times_1 {\rP}_{(1)}\times_2 {\rP}_{(2)}+\ldots\\
+(\bPhi \times_C {\rP}_{(C)}^\perp)\times_1 {\rP}_{(1)}\ldots \times_{C-1} \rP_{(C-1)}.
\end{aligned}
\end{equation}
The terms in this decomposition are mutually orthogonal, as  can be seen from the identities
\begin{equation}
\begin{aligned}
 {\langle\bA \times_i {\rP}_{(i)}^\perp,\bB \times_i {\rP}_{(i)}\rangle_F} & = 
\langle{\rP}_{(i)}^\perp A_{(i)},{\rP}_{(i)} B_{(i)}\rangle_{\ell^2} \\
& = \mbox{tr}\left({\rP}_{(i)}^\perp A_{(i)} B_{(i)}^T {\rP}_{(i)}\right)= 
\mbox{tr}\left({\rP}_{(i)}{\rP}_{(i)}^\perp A_{(i)} B_{(i)}^T \right)=0,
\end{aligned}
\end{equation}
which hold for two tensors $\bA$ and $\bB$ of the same sizes as $\bPhi$ and their unfoldings along the $i$th mode $A_{(i)}$ and $B_{(i)}$, respectively.
Employing this decomposition, the orthogonality property and \eqref{eq:deltai}, we obtain the bound  
\begin{equation}
\begin{aligned}
\|\bPhi-\widehat\bPhi\|_{F}^2&= \|\bPhi \times_1 {\rP}_{(1)}^\perp\|_{F}^2 + \|(\bPhi \times_2 {\rP}_{(2)}^\perp)\times_1 {\rP}_{(1)}\|_{F}^2+\ldots\\
& \qquad\qquad + \|(\bPhi \times_C {\rP}_{(C)}^\perp)\times_1 {\rP}_{(1)}\ldots \times_{C-1} \rP_{(C-1)}\|_{F}^2 \\
&\le \|\bPhi \times_1 {\rP}_{(1)}^\perp\|_{F}^2 + \|\bPhi \times_2 {\rP}_{(2)}^\perp\|_{F}^2+\ldots+ \|\bPhi \times_C {\rP}_{(C)}^\perp\|_{F}^2 \\
&= \|\rP_{(1)}^\perp \Phi^{(1)}\|_{F}^2 + \|\rP_{(2)}^\perp \Phi^{(2)}\|_{F}^2+\ldots+ \|\rP_{(C)}^\perp \Phi^{(C)}\|_{F}^2 \\
&\le \eps_s^2 \sum_{i=1}^C\|\Phi^{(i)}\|_{F}^2 =
C\eps_s^2 \|\bPhi\|_{F}^2,
\end{aligned}
\label{aux530}
\end{equation}
where we used the fact that the Frobenious norm of a tensor equals to the Frobenious norm of any unfolding of it.
Finally, the bound in \eqref{eq:accuracy} follows from \eqref{aux473}, \eqref{aux496}, \eqref{aux510}, and \eqref{aux530}.
\end{proof}

\subsection{Adaptive completion.}

In many applications, and in particular for building  ROM  for parametric dynamical systems with accessible error bounds, one is interested in controlling the completion error, i.e., in ensuring that the estimate 
    \begin{equation}\label{eq:estUpper}
        \|\bPhi-\tbPhi\|_{F}\le\eps\|\bPhi\|_{F},
    \end{equation}
holds for a desired $\eps\ge0$. Such a bound is provided by the result in \eqref{eq:accuracy}. In Algorithm~\ref{Alg1} we can set $\eps_{C}$ and  {have a control of $\eps_q$ depending on the  fitting algorithm used for pointwise tensor completion. However $\eps_s$} is implicitly determined by the the index set $\widetilde{\Omega}_D$.        
This motivates an adaptive slice sampling completion algorithm, which ensures that a proxy of \eqref{eq:estUpper} is fulfilled. For the adaptive completion we define a testing set of indexes ${\Omega}_D^{\rm test}\subset\Omega_D$ and gradually increase the training set  $\widetilde{\Omega}_D$ up to the point when the completed tensor satisfies
   \begin{equation}\label{eq:estTest}
        \|\bPhi|_{\Omega^{\rm test}}-\tbPhi|_{\Omega^{\rm test}}\|_{F}\le\eps\|\bPhi|_{\Omega^{\rm test}}\|_{F},\quad\text{with}~\Omega^{\rm test}=\Omega_{C}\otimes {\Omega}_D^{\rm test}.
    \end{equation}

Assuming that for any given $\bj\in \Omega_D$ we may retrieve a slice $\bPhi_{:,\bj}$ of $\bPhi$, the adaptive completion algorithm can be summarized as follows.

\begin{algorithm}[Adaptive slice sampling tensor completion]
\label{Alg1a}
~\\[1ex]
\textbf{Input:} Testing set
${\Omega}_D^{\rm test}\subset \Omega_D$, initial training set $\widetilde{\Omega}_D\subset \Omega_D$, threshold $\eps$, maximum number of steps $N_{\max}$, and increments $P_i\in\mathbb{N}_+$, $i=1,\ldots,N_{\max}$.
\begin{enumerate}
\item For $\bj\in\widetilde{\Omega}_D$ retrieve slices $\bPhi_{:,\bj}$ to form the training data $\cD$.
\item For $\bj\in{\Omega}_D^{\rm test}$ retrieve slices $\bPhi_{:,\bj}$ to compute $\bPhi|_{\Omega^{\rm test}}$.
\item Let $\eps_{C}=\eps/\sqrt{C}$.
\item For $i = 1,\ldots,N_{\max}$:
\begin{itemize}
\item[(a)] Execute Algorithm 1 with inputs $\widetilde{\Omega}_D$, $\cD$, and $\eps_C$;
\item[(b)] Break if \eqref{eq:estTest} holds;
\item[(c)] Add $P_i$ more indices  to $\widetilde{\Omega}_D$, update $\cD$.
\end{itemize}
\end{enumerate}
\noindent \textbf{Output:} $\rm{HTT} \big( \tbPhi \big)$ and the error $\|\bPhi|_{\Omega^{\rm test}}-\tbPhi|_{\Omega^{\rm test}}\|_{F}/\|\bPhi|_{\Omega^{\rm test}}\|_{F}$.
\end{algorithm}

\begin{remark}[Alternative hybridization]~\\
An alternative hybridized completion strategy can be based on a joint reduced basis for the first $C$ modes. Let
$
\rA=\mbox{matr}_{C}(\bPhi)\in
\mathbb{R}^{|M|\times |K|},$ $
|M|=\prod_{i=1}^C M_i,$ $|K|=\prod_{j=1}^D K_j,
$
and consider an $\eps$-truncated SVD
\[
\rA\approx \rU\Sigma\rV^T,\qquad
\rU\in \mathbb{R}^{|M|\times \widetilde r}.
\]
For $\eps=0$, one has $\widetilde r=\widehat r_C$. If the singular values of $\rA$ decay sufficiently fast so that
$
\widetilde r < \prod_{i=1}^C q_i,
$
then one may project $\bPhi$ onto the joint space-time basis given by the columns of $\rU$. This can be done by vectorizing the first $C$ modes for $\bPhi$ or, equivalently, the columns of $\rU$ may be reshaped back into tensors of size $M_1\times\cdots\times M_C$. This gives $\widetilde r$ coefficient tensors of size $K_1\times\cdots\times K_D$, to which TT-completion can be applied component-wise.

In practice, the matrix $\rU$ can be approximated  by using the right singular vectors of the last-mode unfolding of $\bPhi^P$. This alternative replaces the $\prod_{i=1}^C q_i$ component-wise TT-completion problems in the current HTT construction by $\widetilde r$ such problems, at the price of storing and working with a joint basis in $\mathbb{R}^{|M|}$. We plan to study this alternative approach to hybridization and completion elsewhere.
\end{remark}

\section{ROM for parametric dynamical systems}
\label{sec3}

\subsection{Projection TROM} 

For fast and accurate computations of trajectories $\bu(t,\balpha)$  {solving~\eqref{eqn:GenericPDE}} for any $\balpha\in\cA$, we consider a projection based ROM which uses the proposed HTT LRTD instead of the conventional POD as the dimension reduction technique. The HTT decomposition allows for the recovery of information about the parameter dependence of reduced spaces from a smaller set of pre-computed snapshots. This information is exploited  for building parameter specific ROMs. The LRTD–ROM was introduced in \cite{mamonov2022interpolatory} and further developed and analyzed in \cite{mamonov2024tensorial,mamonov2023analysis}. However, in those earlier works it is assumed that a full snapshot tensor $\bPhi$ is available and therefore its LRTD 
in one of the standard tensor rank revealing formats (e.g., CP, Tuker, and TT) can be computed.

 {For the sake of exposition, we assume that all spatial degrees of freedom are vectorized into a single mode, so that only space and time remain as physical modes; i.e., $C = 2$ for the rest of this section.} In projection TROM we aim to find a low dimensional subspace $V^\ell(\balpha)\in\R^{M_1}$ with $\ell=\dim(V^\ell(\balpha))$,
which is parameter-specific and an approximation  $\bu^{\rm rom}(t,\balpha)$ to $\bu(t,\balpha)$ is found by solving  \eqref{eqn:GenericPDE} projected onto $V^\ell(\balpha)$. This TROM solution $\bu^{\rm rom}(t,\balpha)$ is then given by its coordinates in an orthogonal basis for $V^\ell(\balpha)$.

In turn, the orthogonal basis of $V^\ell(\balpha)$ is recovered by its coordinates in the basis of the \emph{universal reduced space} $\widetilde{U}$, which is the span of all first-mode fibers of the low-rank tensor:
\begin{equation}\label{Univ}
    	\widetilde{U} = \text{range}\big(\widetilde\Phi_{(1)}\big).
\end{equation}
By the construction of $\widetilde\bPhi$, an orthogonal basis for $\widetilde{U}$ is given by the columns of the matrix $\rU^{(1)}$ from \eqref{eq:SVD1}.

In an interpolatory version of the TROM (see \cite{olshanskii2024approximating} for a non-interpolatory version), one finds  $V^\ell(\balpha)$ through interpolation procedure for the slices of $\widetilde\bPhi$. 
To define it,  we first consider an interpolation operator that approximates a smooth function  $g:\cA\to \R$ using its values at the grid nodes $\hcA$. More precisely, we assume  
$\bchi^i \,:\, \cA \to \mathbb{R}^{K_i},\quad i=1,\dots,D,$
such that for  any continuous function $g:\cA\to \R$,
\begin{equation}\label{Interp}
I(g):= \sum_{k_1=1}^{K_1}\dots \sum_{k_D=1}^{K_D}  \left(\bchi^1 (\balpha ) \right)_{k_1}\dots \left(\bchi^D (\balpha ) \right)_{k_D}
g\big(\halpha_1^{k_1},\dots,\halpha_D^{k_D}\big)
\end{equation}
defines an interpolant for
$g$.  One straightforward choice is  the Lagrange interpolation of order $p$: for any  $\balpha \in \cA$, let $\widehat{\alpha}_i^{i_1}, \ldots, \widehat{\alpha}_i^{i_p}$ 
be the $p$ closest grid nodes to $\alpha_i$ on $[\alpha_i^{\min}, \alpha_i^{\max}]$, for $i=1,\ldots,D$.
Then, 
\begin{equation}
	\label{eqn:lagrange}
	\big(\bchi^i (\balpha)\big)_j = 
	\begin{cases} 
		\prod\limits_{\substack{m = 1, \\ m \neq k}}^{p}(\widehat{\alpha}_i^{i_m}-\alpha_i) \Big/ 
		\prod\limits_{\substack{m = 1, \\ m \neq k}}^{p}(\widehat{\alpha}_i^{i_m}-\widehat{\alpha}_i^j), 
		& \text{if } j = i_k \in \{i_1,\ldots,i_p\}, \\
		\qquad\qquad\qquad\qquad\qquad\qquad\qquad 0, & \text{otherwise}, \end{cases}
\end{equation}
are the entries of $\bchi^i (\balpha)$ for $j=1,\dots,K_i$.

Given $\bchi^i$, we introduce the `local' low-rank matrix 
$\widetilde{\Phi} (\balpha)$ via the in-tensor interpolation procedure for tensor $ \widetilde{\bPhi} $:
\begin{equation}
	\label{eqn:extractbt}
	\widetilde{\Phi} (\balpha) = \widetilde{\bPhi} 
	\times_3 \bchi^1(\balpha) \times_4 \bchi^2(\balpha) \dots \times_{D+2} \bchi^D(\balpha) 
	\in \R^{M_1 \times M_2}.
\end{equation}
If $\balpha = \bhalpha \in \hcA$, then $\bchi^i(\hbalpha)$ simply encodes the position of $\widehat{\alpha}_i$ 
among the grid nodes on $[\alpha^{\min}_i, \alpha^{\max}_i]$. Therefore, for $\eps=0$ the  
matrix $\widetilde{\Phi} (\hbalpha)$ is  exactly the matrix of all  
snapshots for the particular $\hbalpha$. For a general $\balpha\in \cA$, the matrix
$\widetilde{\Phi} (\balpha)$ is the result of interpolation between  snapshots (approximately) recovered by solving the completion problem. 
We will show that this interpolation is easy to do  if $\widetilde{\bPhi}$ is in HTT format.
Specifically, for an arbitrary given $\balpha \in \cA$ the parameter-specific local reduced space $V^\ell(\balpha)$ of dimension $\ell$ is the space spanned by the first  $\ell$ left singular vectors of $\widetilde{\Phi} (\balpha)$:
\begin{equation}
\label{Vrom}V^\ell(\balpha)=\mbox{range}\left( [\rS(\balpha)]_{1:\ell} \right),
\end{equation}
where
\begin{equation}
\widetilde{\Phi} (\balpha) = \rS(\balpha) \Sigma(\balpha) \rV(\balpha)^T,
\end{equation}
is the SVD of $\widetilde{\Phi} (\balpha)$.

\subsection{In-tensor interpolation and finding the local basis}

A remarkable fact is that we do \emph{not} need to build the matrix $\widetilde{\Phi} (\balpha)$ and compute its SVD to find an orthogonal basis for $V^\ell(\balpha) $ once $\widetilde{\bPhi}$ is given in the HTT format. 
To see this, assume $\widetilde{\bPhi}$ is given in the HTT format with $C$-ranks $q_1$ and $q_2$ and define the parameter-dependent 
\emph{core matrix} $\rC_\chi (\balpha) \in \R^{q_1 \times q_2}$ as follows
\begin{equation}
\label{eqn:C_TT}
[\rC_\chi (\balpha)]_{i_1,i_2} =\widetilde\bPhi_\bi^q \times_1 \bchi^1 (\balpha)\dots \times_D \bchi^D (\balpha),\quad\text{with}~\bi=(i_1,i_2).
\end{equation} 
Practically, the tensor matrix products in \eqref{eqn:C_TT} are computed as follows. Each $\bg^\bi_j$  from \eqref{eqn:bphiqitt} is a $\tr^\bi_{j-1}\times K_j\times \tr^\bi_j$ tensor, then  $G^\bi_j=\bg^\bi_j\times_2\bchi^j (\balpha)$ is a $\tr^\bi_{j-1}\times \tr^\bi_j$ matrix and we calculate $[\rC_\chi (\balpha)]_{i_1,i_2}=G^\bi_1\dots G^\bi_D$ as a product of $D$ small size matrices. 

By the definition of $\widetilde{\Phi} (\balpha)$ and \eqref{eq:tPhi} we have
\begin{equation}
	\label{eqn:aux475}
 \begin{split}
	\widetilde{\Phi} (\balpha) &= \widetilde{\bPhi}\times_{3} \bchi^1 (\balpha)\dots\times_{D+2} \bchi^D (\balpha)\\
 &= (\widetilde{\bPhi}^q\times_1 
\widetilde{\rU}^{(1)}\times_2 
\widetilde{\rU}^{(2)}) \times_{3} \bchi^1 (\balpha)\dots\times_{D+2} \bchi^D (\balpha)\\
 &= (\widetilde{\bPhi}^q\times_{3} \bchi^1 (\balpha)\dots\times_{D+2} \bchi^D (\balpha)) \times_1 
\widetilde{\rU}^{(1)}\times_2 
\widetilde{\rU}^{(2)} \\
&= \rC_\chi (\balpha) \times_1 
\widetilde{\rU}^{(1)}\times_2 
\widetilde{\rU}^{(2)} =    \widetilde{\rU}^{(1)}   \rC_\chi (\balpha) (\widetilde{\rU}^{(2)})^T.
 \end{split}
\end{equation}
We combine this representation of $\widetilde{\Phi} (\balpha)$
with the SVD of the core matrix
\begin{equation}
 \rC_\chi (\balpha)  = \rU_c \Sigma_c \rV_c^T
\label{eqn:coresvdtt}
\end{equation}
to obtain
\begin{equation}
\widetilde{\Phi}(\balpha) = 
\left(\widetilde{\rU}^{(1)}\rU_c\right)\Sigma_c \left(\widetilde{\rU}^{(2)}\rV_c\right)^T.
\label{eqn:Phie_TT}
\end{equation}
The right-hand side of \eqref{eqn:Phie_TT} is the thin SVD of $\widetilde{\Phi}(\balpha)$, 
since the matrices $\widetilde{\rU}^{(1)}$, $\rU_c$, $\widetilde{\rU}^{(2)}$, and $\rV_c$ are all orthogonal. We conclude that 
\emph{the coordinates $\left\{\bbeta_1(\balpha),\dots,\bbeta_\ell(\balpha)\right\}$ of the local reduced basis 
in the orthogonal basis of the universal space $\widetilde{U}$ are the first $\ell$ columns of $\rU_c$.} The parameter-specific ROM basis is then 
$\{ \bz_i(\balpha) \}_{i=1}^{\ell}$, with $\bz_i(\balpha)=\widetilde{\rU}^{(1)}\bbeta_i(\balpha)$.

Note that the explicit computation of the basis functions $\bz_i(\balpha)$ is \emph{not} required during the online stage.
 {In particular, the projection of the dynamical system onto the space $V^\ell(\balpha) = \mathrm{span}\{\bz_i(\balpha)\}_{i=1}^{\ell}$ is performed using only the matrix of basis coordinates $\widetilde{\rU}^{(1)}$; see Algorithm~\ref{alg:CTROM}.}
Therefore, the essential information about $\widetilde{\bPhi}$ needed for the online part includes only the set of TT-tensors $\widetilde\bPhi_\bi$:
\begin{equation}
\label{eqn:coreTT}
\mbox{online}(\widetilde{\bPhi}) = 
\left\{\widetilde\bPhi_\bi,~~\bi\in\Omega^q_{C}
\right\}.
\end{equation}

\subsection{Summary of the completion based TROM}

Here we summarize the two stage completion based TROM (CTROM) approach. The core of the offline stage is finding a low-rank approximation of the snapshot tensor in the HTT format through 
the adaptive slice sampling completion Algorithm~\ref{Alg1a}. Computing one slice of $\bPhi$ amounts to integrating \eqref{eqn:GenericPDE} numerically for a given fixed parameter $\hbalpha\in\hcA$, a task performed by the numerical solver 
$\cS: \balpha \to \Phi(\balpha)$. The second step of the offline stage is to project the system \eqref{eqn:GenericPDE} 
onto the universal space $\widetilde{U}$, as defined in \eqref{Univ}. The orthogonal basis for $\widetilde{U}$ is stored offline, while the projected system is passed to online stage. The particular computational details of the projection process depend on both the system form and its numerical solver $\cS$. Detailed examples of system projections can be found \cite{mamonov2022interpolatory} for the two particular systems studied in Section~\ref{s:num}.

At the online stage, for any parameter $\balpha\in\cA$ one computes the orthogonal basis of the local reduced space $V^\ell(\balpha)$ represented by its coordinates in $\widetilde{U}$ and the system is projected the second time onto $V^\ell(\balpha)$. The computations at the online stage operate only with objects of reduced dimensions. 

We summarize the two stage CTROM approach outlined above in Algorithm~\ref{alg:CTROM}.

\begin{algorithm}[Completion based TROM]
\label{alg:CTROM} 
~
\begin{itemize}
\item \textbf{Offline stage}.\\
\textbf{Input:} 
\begin{itemize}
\item Numerical solver for \eqref{eqn:GenericPDE} that computes the snapshot matrix \eqref{eqn:phialpha} for a given $\balpha \in \cA$, i.e., $\cS:\balpha \to \Phi(\balpha)$;
\item Parameter grid $\hcA$ as in \eqref{eqn:grid}; 
\item Target accuracy $\eps>0$;
\item Testing set ${\Omega}_D^{\rm test} \subset \Omega_D$, initial training set $\widetilde{\Omega}_D \subset \Omega_D$, maximum number of steps $N_{\max}$, increments $P_i\in\mathbb{N}_+$, $i=1,\ldots,N_{\max}$.
\end{itemize}
\begin{enumerate}
\item  Execute Algorithm~\ref{Alg1a} with inputs ${\Omega}_D^{\rm test}$, $\widetilde{\Omega}_D$, $\eps$, $N_{\max}$, and $P_i\in\mathbb{N}_+$, $i=1,\ldots,N_{\max}$, to compute slice sampling completion $\rm{HTT} \big(\tbPhi \big)$
using the numerical solver $\cS$ in Step 4(c) for updating the data $\cD$;
\item Project the system \eqref{eqn:GenericPDE} onto the universal ROM space $\widetilde{U}$ using the orthonormal basis
stored in $\widetilde{\rU}^{(1)}$ from $\rm{HTT} \big( \tbPhi \big)$ to obtain the projected solver $\widetilde{\cS}$ from the full order solver $\cS$;
\end{enumerate}
\textbf{Output:} $\rm{HTT} \big( \widetilde{\bPhi} \big)$,
projected solver $\widetilde{\cS}$.
\item \textbf{Online stage}. \\
\textbf{Input:} $\mbox{\rm online}(\widetilde{\bPhi})$ as defined in \eqref{eqn:coreTT}, 
reduced space dimension 
$\ell \le q_{\min} = \min\{q_1,q_2\}$, 
parameter vector $\balpha \in \cA$, and projected solver $\widetilde{\cS}$;
\begin{enumerate}
\item Use tensors $\widetilde\bPhi_i$ from $\mbox{\rm online}(\widetilde{\bPhi})$ to assemble the core matrix 
$\rC_\chi (\balpha) \in \R^{q_1 \times q_{2}}$ as in \eqref{eqn:C_TT};
\item Compute the SVD of the core matrix $\rC_\chi (\balpha) = \rU_c \Sigma_c \rV_c^T$ with $\rU_c = [\widetilde{\bu}_1, \ldots, \widetilde{\bu}_{q_{\min}}]$;
\item Set $\bbeta_i(\balpha) =\widetilde{\bu}_i$, {\small$i = 1,\ldots,\ell$};
\item Project the solver $\widetilde{\cS}$ onto the local reduced basis with coordinates $\{ \bbeta_i(\balpha) \}_{i=1}^{\ell}$ to obtain the parameter-specific CTROM solver $\widetilde{S}_{\ell}(\balpha)$.
\end{enumerate}
\textbf{Output:} Coordinates of the reduced basis in $\widetilde{U}$:
$\{ \bbeta_i(\balpha) \}_{i=1}^{\ell} \subset \mathbb{R}^{q_1}$ and the CTROM solver $\widetilde{S}_{\ell}(\balpha)$.\\
\end{itemize}
\end{algorithm}

 {
\subsection{Complexity of CTROM}
The computational requirements of the completion-based TROM are determined by the following steps:  
\begin{itemize}
	\item[(i)] \textbf{Offline stage}: Completion of $\widetilde{\bPhi}$ in the HTT format; 
	\item[(ii)] Transfer of the $\mbox{online}(\widetilde{\bPhi})$ component of the compressed tensor to the online stage;
	\item[(iii)] \textbf{Online stage}: Computation of the coordinates of the parameter-specific reduced basis for a given $\balpha$;
	\item[(iv)] Solution of \eqref{eqn:GenericPDE} projected onto the reduced space.
\end{itemize}

Since step (iv) is common to all projection-based ROMs, we focus on steps (i)--(iii). The offline stage is dominated by the completion algorithm and repeated FOM solves used to generate the training set. The number of FOM calls scales at least as $O(\sum_{j=1}^D r^{\rm max}_j K_j r^{\rm max}_{j+1})$, where $r^{\rm max}_j=\max_{\bi\in\Omega_C^q}r^\bi_j$. The cost of the completion step depends on the chosen scheme; in the numerical experiments we employ the stable rank-adaptive ALS (SALSA) method from~\cite{grasedyck2019stable}, with  complexity $O(\sum_{j=1}^D |\Omega_C^q| |\widetilde{\Omega}_D| |r^{\rm max}_j r^{\rm max}_{j+1}|^2)$
per adaptive step.  

The quantity $\mbox{online}(\widetilde{\bPhi})$, defined in \eqref{eqn:coreTT}, determines the information communicated to the online stage in step (ii), which amounts to $O(|\Omega_C^q|\sum_{j=1}^D r^{\rm max}_j K_j r^{\rm max}_{j+1})$ degrees of freedom, plus the projected operator $\widetilde{S}$.  

In step (iii), the cost of computing the $\balpha$-specific reduced basis is governed by the interpolation procedure and the evaluation of the leading $n$ left singular vectors of the core matrix. Since the vectors $\be^i(\balpha)$ are sparse (typically $p=2$ or $3$ nonzero entries in Cartesian sampling), forming the core matrix $\rC(\balpha)$ requires  
$
O \Big( \sum_{i=2}^{D} \widetilde{r}_{i-1}\widetilde{r}_{i}\widetilde{r}_{i+1} \Big)
$  
operations. The subsequent SVD involves only the small $\widetilde{r}_{1}\times \widetilde{r}_{D+1}$ matrix. If a reduced basis in the physical space is required, the vectors are recovered as linear combinations of the columns of $\rU$, with cost $O(|\Omega_C^q|\widetilde{r}_1\ell)$.  

}

\section{Numerical examples} 
\label{s:num} \label{sec4} 

We assess the performance of the sliced sampled completion and the completion based tensor ROM numerically in two examples of parametric dynamical systems. 
To address the component-wise completion problems for $\bPhi^q_\bi$, we use SALSA method in the TT format from~\cite{grasedyck2019stable}, which ensures \eqref{eq:estTest} with prescribed $\varepsilon$.  All other input parameters of SALSA are left at default values. The corresponding compression $D$-ranks are determined adaptively during the completion process.
We note that SALSA method does not necessarily produce a \textit{minimal}-rank TT tensor that fits the data within the prescribed tolerance, and thus it can be viewed as an approximate solution to the completion problem~\eqref{CProblem2}.

The code of Matlab implementation of CTROM used for the numerical experiments is available upon request.

\subsection{Parameterized heat equation}
\label{sec:heat}

The first example is a dynamical system corresponding to the heat equation
\begin{equation}
u_t(\bx, t, \balpha) = \Delta u(\bx, t, \balpha), \quad \bx \in \Omega, \quad t \in (0, T)
\label{eqn:heat}
\end{equation}
in a rectangular domain with three holes 
$\Omega = \Omega_r \setminus (\Omega_1 \cup \Omega_2 \cup \Omega_3)$ with $\Omega_r = [0,10] \times [0,4]$, as 
shown in Figure~\ref{fig:dom3hole}. Zero initial condition is enforced and the terminal time is set to $T = 20$.
The system has $D = 4$ parameters that enter the boundary conditions:
\begin{eqnarray}
\left. (\bn \cdot \nabla u + \alpha_1(u-1)\,) \right|_{\Gamma_o} & = & 0,
\label{eqn:bco} \\
\left. \left( \bn \cdot \nabla u + \frac{1}{2} u \right) \right|_{\partial \Omega_j} & =&  \frac{1}{2} \alpha_{j+1},
\quad j=1,2,3, \label{eqn:bcj} \\
\left. (\bn \cdot \nabla u) \right|_{\partial \Omega_r \setminus \Gamma_o} & = & 0,
\label{eqn:bcins}
\end{eqnarray}
where, $\bn$ is the outer unit normal and $\Gamma_o = 0 \times [0,4]$ is the left side of the rectangle $\Omega_r$.
The parameter domain is the box $\cA = [0.01, 0.501] \times [0, 0.9]^3$. The system \eqref{eqn:heat}--\eqref{eqn:bcins} is discretized with $P_2$ finite elements on a quasi-uniform 
triangulation of $\Omega$ with maximum element size $h$.

\begin{figure}[ht]
\begin{center}
\includegraphics[width=0.5\textwidth]{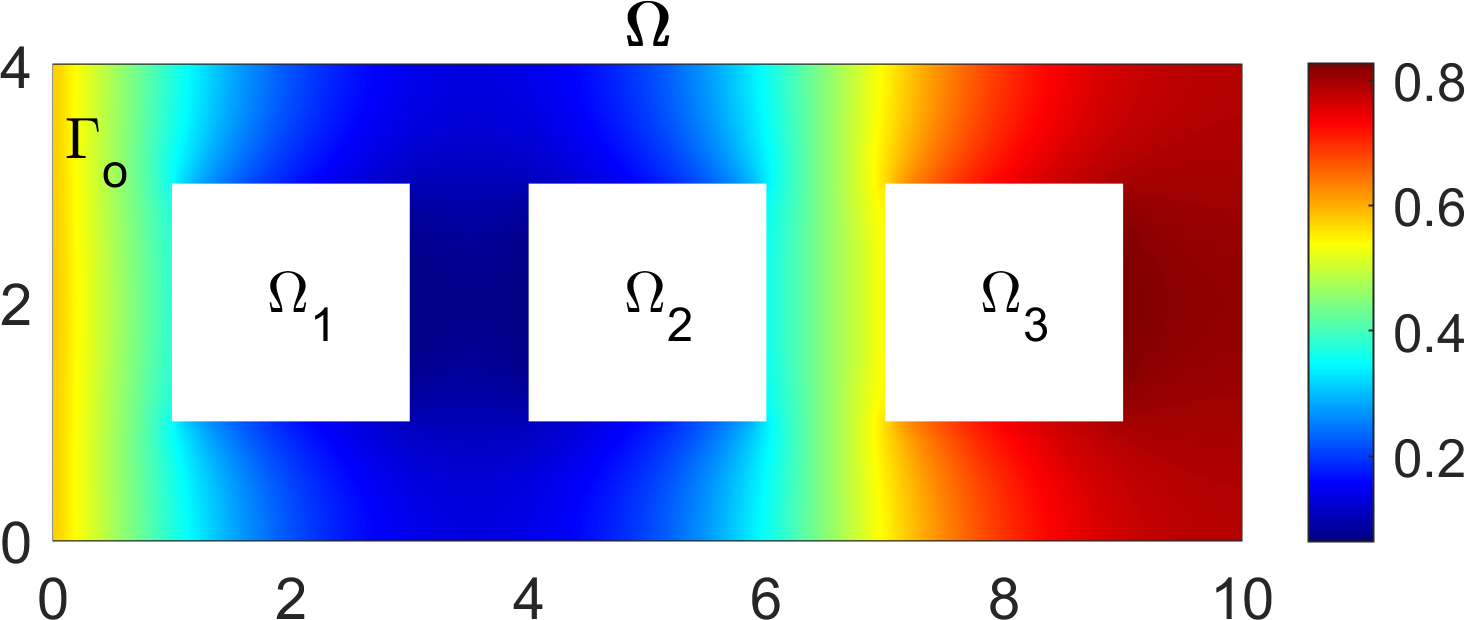} 
\end{center}
\caption{Domain $\Omega$ and the solution $u(\bx, T, \balpha)$ of \eqref{eqn:heat}--\eqref{eqn:bcins} corresponding to $\balpha = (0.5, 0, 0, 0.9)^T$.}
\label{fig:dom3hole}
\end{figure}

In the numerical experiments below we assess the performance of CTROM for the system \eqref{eqn:heat}--\eqref{eqn:bcins}. For this purpose we employ the discrete $L^2(0,T;L^2(\Omega))$ norm error for a prescribed parameter value $\balpha$, denoted by 
\begin{equation}
E_{\balpha} = \Delta t \sum_{n=1}^{N} \left\| u(\;\cdot\;,t_n,\balpha) -  u_\ell(\;\cdot\;,t_n,\balpha) \right\|^2_{L^2(\Omega)},
\label{eqn:erra}
\end{equation}
where $u$ is the true (numerical) solution of \eqref{eqn:heat}--\eqref{eqn:bcins}, while $u_\ell$ is the CTROM solution.
In order to evaluate CTROM accuracy over the whole parameter domain $\cA$ we choose a testing set ${\cB} = \{ \widetilde{\balpha}_k \}_{k=1}^{\widetilde{K}} \subset \cA$ and compute the following maximum and mean quantities
\begin{eqnarray}
{E}_{\max} & = & 
\left( \max_{k \in \{1,2,\ldots,\widetilde{K}\} } E_{\widetilde{\balpha}_k} \right)^{1/2}, 
\label{eqn:errestmax} \\
{E}_{\text{mean}} & = & 
\left( \frac{1}{\widetilde{K}} \sum_{k=1}^{\widetilde{K}} E_{\widetilde{\balpha}_k} \right)^{1/2}.
\label{eqn:errestmean}
\end{eqnarray}
To form ${\cB}$ we choose $\widetilde{\balpha}_k \in \cA$ at random. 

\begin{table}
\begin{center}
\begin{tabular}{c|c|c|c|c|c}
\hline
Experiment $\#$ & $1$ & $2$ & $3$ & $4$ & $5$ \\
\hline
$h$ & $0.33$ & $0.2$ & $0.1$ & $0.2$ & $0.2$ \\
\hline
$[K_1,K_2,K_3,K_4]$ & $[8,5,5,5]$ & $[8,5,5,5]$ & $[8,5,5,5]$ & $[10,7,7,7]$ & $[12,9,9,9]$ \\
\hline\\[-2.3ex]
$|\widetilde{\Omega}_D| / |\Omega_D|$ & $0.22$ & $0.22$ & $0.22$ & $0.094$ & $0.048$\\
\hline
$[q_1, q_2]$ & $[51, 18]$ & $[51, 18]$ & $[51, 18]$ & $[50, 18]$ & $[51, 18]$ \\
\hline
$[r_1^{\max}, r_2^{\max}, r_3^{\max}]$ & $[7, 8, 5]$ & $[7, 7, 5]$ & $[7, 8, 5]$ & $[8, 7, 7]$ & $[7, 7, 7]$ \\
\hline
$r_1^{\text{mean}}$ & $6.002$ & $6.003$ & $6.025$ & $6.008$ & $6.005$ \\
\hline
$r_2^{\text{mean}}$ & $5.012$ & $5.022$ & $5.085$ & $5.020$ & $5.015$ \\
\hline
$r_3^{\text{mean}}$ & $4.014$ & $4.026$ & $4.090$ & $4.046$ & $4.042$ \\
\hline
$\ell$ & $18$ & $18$ & $18$ & $18$ & $18$ \\
\hline
${E}_{\max}$ & $2.729 \cdot 10^{-3}$ & $2.764 \cdot 10^{-3}$ & $2.779 \cdot 10^{-3}$ & $1.256 \cdot 10^{-3}$ & $4.937 \cdot 10^{-4}$ \\
\hline
${E}_{\text{mean}}$ & $1.329 \cdot 10^{-3}$ & $1.331 \cdot 10^{-3}$ & $1.338 \cdot 10^{-3}$ & $4.512 \cdot 10^{-4}$ & $2.484 \cdot 10^{-4}$ \\
\hline
\end{tabular}
\end{center}
\caption{Results of numerical experiments for the heat equation.\label{tab:heat}}
\end{table}

We perform five numerical experiments, numbered $1-5$ with parameters and results reported in Table~\ref{tab:heat}. The following quantities were used in all experiments: $N = 100$ time-domain snapshots were computed; threshold value $\eps = \sqrt{2} \cdot 10^{-6}$ was set in Algorithm~\ref{Alg1a} so that the corresponding threshold in Algorithm~\ref{Alg1} takes the value of $\eps_C = 10^{-6}$. The rank-adaptive componenent-wise TT completion for $\bPhi^q_i$, $\bi\in\Omega_C^q$ was ran to ensure \eqref{eq:epsq} with $\varepsilon_q=10^{-6}$.  

For each experiment we report the
$C$-ranks $[q_1, q_2]$ from step 2(c) of Algorithm~\ref{Alg1} as well as the following quantities from the the last iteration of Algorithm~\ref{Alg1a}: the sampling rate $|\widetilde{\Omega}_D|/|\Omega_D|$, the maximum and mean $D$-ranks 
\begin{equation}
r_k^{\max} = \max\limits_{\bi \in \Omega_C^q} r_{k}^{\bi}, \quad
r_k^{\text{mean}} = \frac{1}{|\Omega_C^q|}\sum_{\bi \in \Omega_C^q} r_{k}^{\bi}, \quad
k = 1,2,\ldots,D-1,
\label{eqn:rmaxmean}
\end{equation}
from TT completions in step 3(b) of Algorithm~\ref{Alg1}.

In experiments 1-3 the sampling set $\widehat\cA$ was fixed with
$K_1=8$, $K_2 = 5$, $K_3 = 5$ and $K_4 = 5$ grid nodes in each parameter direction, respectively. At the same time, the spatial FEM discretization mesh was refined with the maximum element size decreasing from $h = 0.33$ to $h = 0.2$ and then to $h=0.1$. We observe in the corresponding columns of Table~\ref{tab:heat} that CTROM is robust with respect to the mesh refinement in all of the reported quantities. This includes the $C$-ranks $[q_1, q_2]$, the maximum and mean $D$-ranks \eqref{eqn:rmaxmean}, as well as both error quantities \eqref{eqn:errestmax}--\eqref{eqn:errestmean} that barely change during mesh refinement.

In experiments 2, 4, 5 the FEM mesh is kept the same with $h = 0.2$, but the 
set $\widehat\cA$ is refined from $8\times5^3 = 1000$ to $10\times7^3 = 3430$ and then to $12\times9^3 =8748$ grid nodes, respectively. The purpose of this set of experiments is to show that the sampling rate required for achieving the same level of accuracy decreases as the size of $\widehat\cA$ grows. Indeed, the sampling rate $|\widetilde{\Omega}_D| / |\Omega_D|$ first decreases from $0.22$ to $0.094$ and then to $0.048$. At the same time, the $C$-ranks $[q_1, q_2]$ and the maximum and mean $D$-ranks \eqref{eqn:rmaxmean} remain stable, while the error quantities \eqref{eqn:errestmax}--\eqref{eqn:errestmean} improve (decrease). This demonstrates that slice sampling tensor completion alleviates the curse of dimensionality in that the sampling rate can be decreased for larger $\widehat\cA$ without the loss of accuracy of the resulting CTROM.

 {
In the final experiment for the parameterized heat problem, we investigate }
\begin{wrapfigure}{r}{0.39\textwidth}
	\centering \vskip-4ex
	\includegraphics[width=0.38\textwidth]{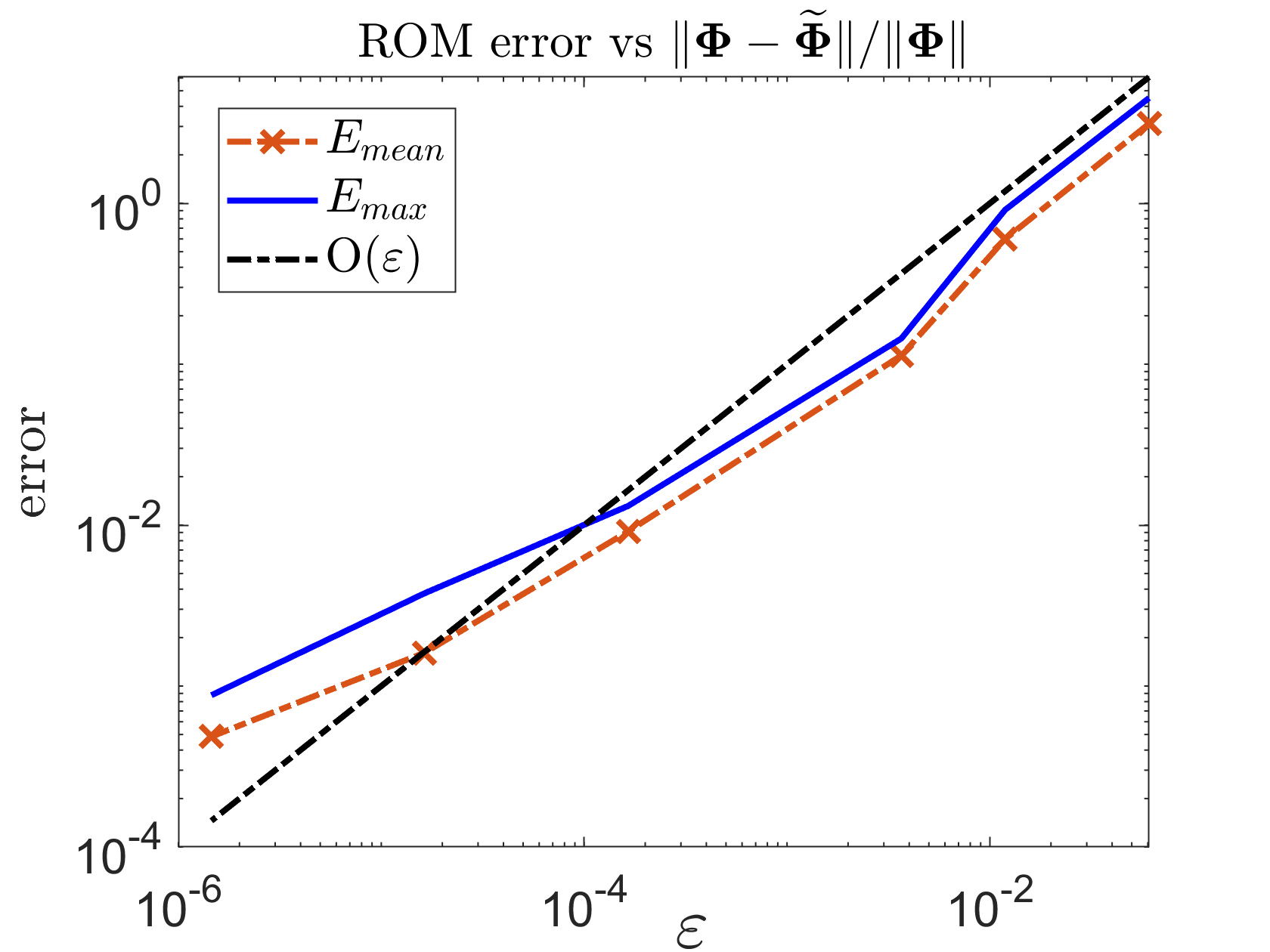}
	\caption{CTROM solution error versus completion accuracy.}
	\label{fig:Err}
\end{wrapfigure}
 {the dependence of the ROM solution error on the completion accuracy. The ROM solution error is defined in equations~\eqref{eqn:errestmax}--\eqref{eqn:errestmean}.
The completion error is characterized by $\varepsilon$ in \eqref{eq:estUpper}, and the estimate \eqref{eq:accuracy} provides an upper bound for it.
Since we do not have direct access to $\varepsilon$ or to certain quantities appearing in \eqref{eq:accuracy}, the completion accuracy is instead assessed on a test set, as described in \eqref{eq:estTest}.
The results presented in Fig.~\ref{fig:Err} show that the ROM errors decrease almost linearly over the given range of $\varepsilon$.
In this experiment, we set $h = 0.2$, used a $10 \times 7^3$ mesh in $\widehat{\mathcal{A}}$, and chose $\ell = q_2$.
}

 
\subsection{Parameterized advection-diffusion equation}
The second numerical example involves more (up to $D_{\max}=12$) parameters than that in Section~\ref{sec:heat}
and corresponds to the dynamical system resulting from the discretization of a linear advection-diffusion equation
\begin{equation}
u_t(\bx, t, \balpha) = \nu \Delta u(\bx, t, \balpha) - \bseta (\bx, \balpha) \cdot \nabla u(\bx, t, \balpha) + f(\bx),
 \quad \bx \in \Omega, \quad t \in (0, T)
\label{eqn:advdiff}
\end{equation}
in the unit square domain $\Omega = [0,1] \times [0,1] \subset \R^2$, $\bx = (x_1, x_2)^T \in \Omega$,
with final time $T=1$.
Here $\nu = 1/30$ is the diffusion coefficient, $\bseta: \Omega \times \cA \to \R^2$ is the parameterized  advection field 
and $f(\bx)$ is a Gaussian source
\begin{equation}
f(\bx) = \frac{1}{2 \pi \sigma_s^2} \exp \left( - \cfrac{ (x_1 - x_1^s)^2 + (x_2 - x_2^s)^2}{2 \sigma_s^2} \right),
\label{eqn:advdiffsrc}
\end{equation}
with $\sigma_s = 0.05$, $x_1^s = {x_2^s} = 0.25$. 
Homogeneous Neumann boundary conditions and zero initial condition are imposed
\begin{equation}
\left. \left( \bn \cdot \nabla u \right) \right|_{\partial \Omega} = 0, \quad u(\bx, 0, \balpha) = 0.
\label{eqn:advdiffbcic}
\end{equation}
The model parameters enter the divergence free advection field $\bseta$ defined as
\begin{equation}
\bseta(\bx, \balpha) = \begin{pmatrix} \eta_1(\bx, \balpha) \\ \eta_2(\bx, \balpha) \end{pmatrix}  = 
\begin{pmatrix} \cos(\alpha_1) \\ \sin(\alpha_1) \end{pmatrix} 
+ \frac{1}{\pi} \begin{pmatrix} \partial_{x_2} h(\bx, \balpha) \\ - \partial_{x_1} h(\bx, \balpha) \end{pmatrix},
\label{eqn:etacos}
\end{equation}
where $h(\bx)$ is the cosine trigonometric polynomial
\begin{equation}
\begin{split}
h(\bx, \balpha) = & \;\;\;\;
\alpha_2 \cos(\pi x_1) + 
\alpha_3 \cos(\pi x_2) + 
\alpha_4 \cos(\pi x_1) \cos(\pi x_2) \\
& + \alpha_5 \cos(2\pi x_1) + 
\alpha_6 \cos(2\pi x_2) \\
& + \alpha_7 \cos(2\pi x_1) \cos(\pi x_2) +
\alpha_8 \cos(\pi x_1) \cos(2\pi x_2) \\
& + \alpha_9 \cos(2\pi x_1) \cos(2\pi x_2)  + \alpha_{10} \cos(3\pi x_1) + \alpha_{11} \cos(3\pi x_2) \\
& + \alpha_{12} \cos(3 \pi x_1) \cos(\pi x_2)
.
\end{split}
\end{equation}
The system \eqref{eqn:advdiff}--\eqref{eqn:advdiffbcic} is discretized similarly to \eqref{eqn:heat}--\eqref{eqn:bcins}, but using a uniform grid in $\Omega$.

We perform the numerical experiments for the advection-diffusion system for varying number of parameters $D$ by setting $\alpha_j = 0$ for $j=D+1,\ldots,12$. Then, the vectors $[\alpha_1,\ldots,\alpha_D]^T$ belong to the parameter domain that is the box $\mathcal{A} = [0.1\pi, 0.3\pi] \times [-0.1, 0.1]^{D-1}$.

\begin{table}
\begin{center}
\begin{tabular}{c|c|c|c}
\hline
$D$ & $6$ & $9$ & $12$ \\
\hline\\[-2.3ex]
$h$ & $\frac1{32}$ & $\frac1{32}$ & $\frac1{32}$ \\[0.5ex]
\hline
$K_1$ & $10$ & $10$ & $10$ \\
\hline
$K_j$, $j=2,\ldots,D$ & $5$ & $5$ & $5$ \\
\hline\\[-2.3ex]
$|\widetilde{\Omega}_D| / |\Omega_D|$ & $0.05056$ & $0.00204$ & $6\times10^{-5}$ \\
\hline
$[q_1, q_2]$ & $[55, 11]$ & $[66, 11]$ & $[83, 11]$  \\
\hline
$\ell$ & $11$ & $11$ & $11$ \\
\hline
${E}_{\max}$ & $1.890 \cdot 10^{-4}$ & $2.022 \cdot 10^{-4}$ & -- \\
\hline
${E}_{\text{mean}}$ & $1.242 \cdot 10^{-4}$ & $1.334 \cdot 10^{-4}$ & -- \\
\hline
\end{tabular}
\end{center}
\caption{Results of numerical experiments for the advection-diffusion equation with $D=6,9,12$. \label{tab:advdiffD6a}}
\end{table}

We set $\eps = 5 \cdot 10^{-4}$ in  Algorithm~\ref{Alg1a}  and $\eps_C = 10^{-6}$ in Algorithm~\ref{Alg1}. The rank-adaptive componenent-wise TT completion for $\bPhi^q_i$, $\bi\in\Omega_C^q$ was performed to ensure \eqref{eq:epsq} with $\varepsilon_q=10^{-4}$. 
The results of numerical experiments with parametrized advection--diffusion equation are summarized in Tables~\ref{tab:advdiffD6a}--\ref{tab:advdiffD6b} and Figures~\ref{fig:Dall}--\ref{fig:D6b}.

\begin{table}
\begin{center}
\begin{tabular}{c|c|c|c|c|c|c}
\hline
Experiment $\#$ & $1$ & $2$ & $3$ & $4$ & $5$ & $6$\\
\hline\\[-2.3ex]
$h$ & $\frac1{32}$ & $\frac1{32}$ & $\frac1{32}$ & $\frac1{32}$ & $\frac1{64}$ & $\frac1{16}$ \\[0.5ex]
\hline
$K_1$ & $10$ & $19$ & $29$ & $39$ & $10$ & $10$ \\
\hline
$K_j$,  $j=2,\ldots,6$ & $5$ & $9$ & $14$ & $19$ & $5$ & $5$ \\
\hline\\[-2.3ex]
$|\widetilde{\Omega}_D| / |\Omega_D|$ & $0.05056$ & $0.00262$ & $0.00038$ & $0.00008$ & $0.05056$ & $0.05056$\\
\hline
$[q_1, q_2]$ & $[55, 11]$ & $[54, 11]$ & $[53, 11]$ & $[53, 11]$ & $[55, 11]$ & $[56, 11]$ \\
\hline
$r_1^{\text{mean}}$ & $8.500$ & $7.173$ & $7.759$ & $7.593$ & $8.540$ & $8.553$ \\
\hline
$r_2^{\text{mean}}$ & $12.800$ & $10.227$ & $11.855$ & $11.236$ & $12.727$ & $12.823$ \\
\hline
$r_3^{\text{mean}}$ & $15.014$ & $12.144$ & $13.065$ & $13.264$ & $15.011$ & $14.995$ \\
\hline
$r_4^{\text{mean}}$ & $15.805$ & $18.075$ & $19.807$ & $20.001$ & $15.809$ & $15.797$ \\
\hline
$r_5^{\text{mean}}$ & $5.000$ & $9.000$ & $12.380$ & $13.121$ & $5.000$ & $5.000$ \\
\hline
$\ell$ & $11$ & $11$ & $11$ & $11$ & $11$ & $11$ \\
\hline
${E}_{\max}$ & $1.890 \cdot $ & $2.364 \cdot $ & $1.644 \cdot $ & $2.274 \cdot $ & $2.277 \cdot $ & $1.653 \cdot $ \\
& $\cdot 10^{-4}$ & $\cdot 10^{-4}$ & $\cdot 10^{-4}$ & $\cdot 10^{-4}$ & $\cdot 10^{-4}$ & $\cdot 10^{-4}$ \\
\hline
${E}_{\text{mean}}$ & $1.242 \cdot $ & $1.025 \cdot $ & $9.080 \cdot $ & $9.533 \cdot $ & $1.493 \cdot $ & $1.093 \cdot $ \\
& $\cdot 10^{-4}$ & $\cdot 10^{-4}$ & $\cdot 10^{-5}$ & $\cdot 10^{-5}$ & $\cdot 10^{-4}$ & $\cdot 10^{-4}$ \\
\hline
\end{tabular}
\end{center}
\caption{Results of numerical experiments for the advection-diffusion equation with $D=6$. \label{tab:advdiffD6b}}
\end{table}

In the first series of experiments, we vary the parameter space dimension $D$, while keeping the discretization parameters fixed:    $h=1/32$, $\Delta t=\frac1{200}$. The grid in the parameter domain is the Cartesian product of uniform grids for each parameter. We use 10 uniformly distributed nodes for $\alpha_1$ and 5 nodes for $\alpha_i$, $i>1$. The recovered $C$-ranks are given in Table~\ref{tab:advdiffD6a} and the mean $D$-ranks were found to be 
[1, 8.50, 12.8, 15.0,15.8, 5.0, 1] for $D=6$, 
[1, 8.27, 14.2, 18.0, 20.8, 27.0, 29.7, 20.1, 5.0, 1] for $D=9$ and 
[1, 7.54, 12.6, 16.7, 20.8, 26.9, 30.6, 38.6, 48.9, 50.9, 25.0, 5.0, 1] for $D=12$. 
We see that both ranks increase with $D$, reflecting the growing variability of the solution for more complex advection fields as more parameters are involved. For $D = 12$, however, we were only able to achieve a completion accuracy of $\epsilon = 10^{-3}$, which required fitting approximately $27\cdot10^3$ parameters from the training set. The adaptive completion algorithm we used was not efficient for a larger number of training parameters in this example.

\begin{figure}[ht]
\begin{center}
\includegraphics[width=0.32\textwidth]{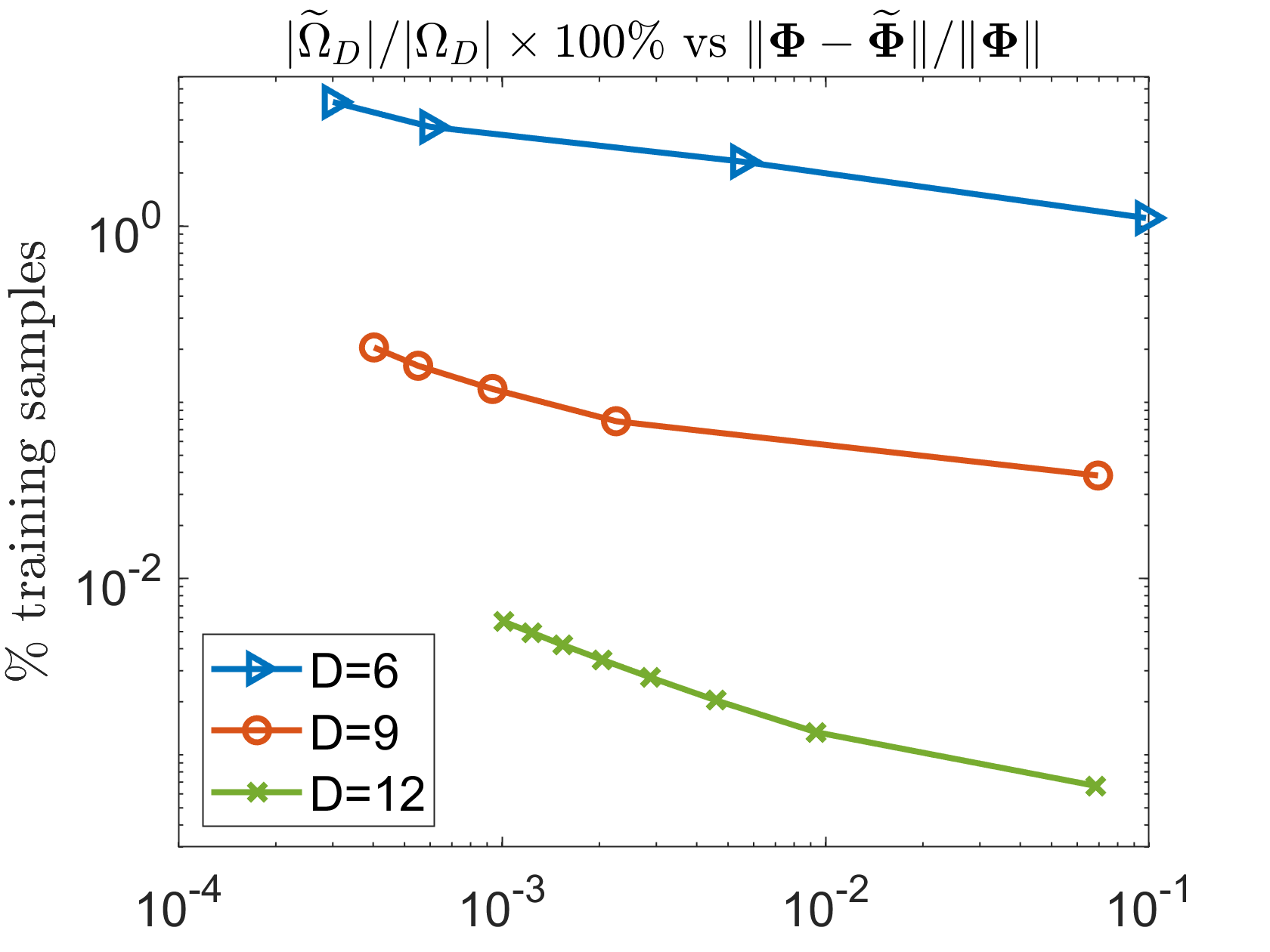} 
\includegraphics[width=0.32\textwidth]{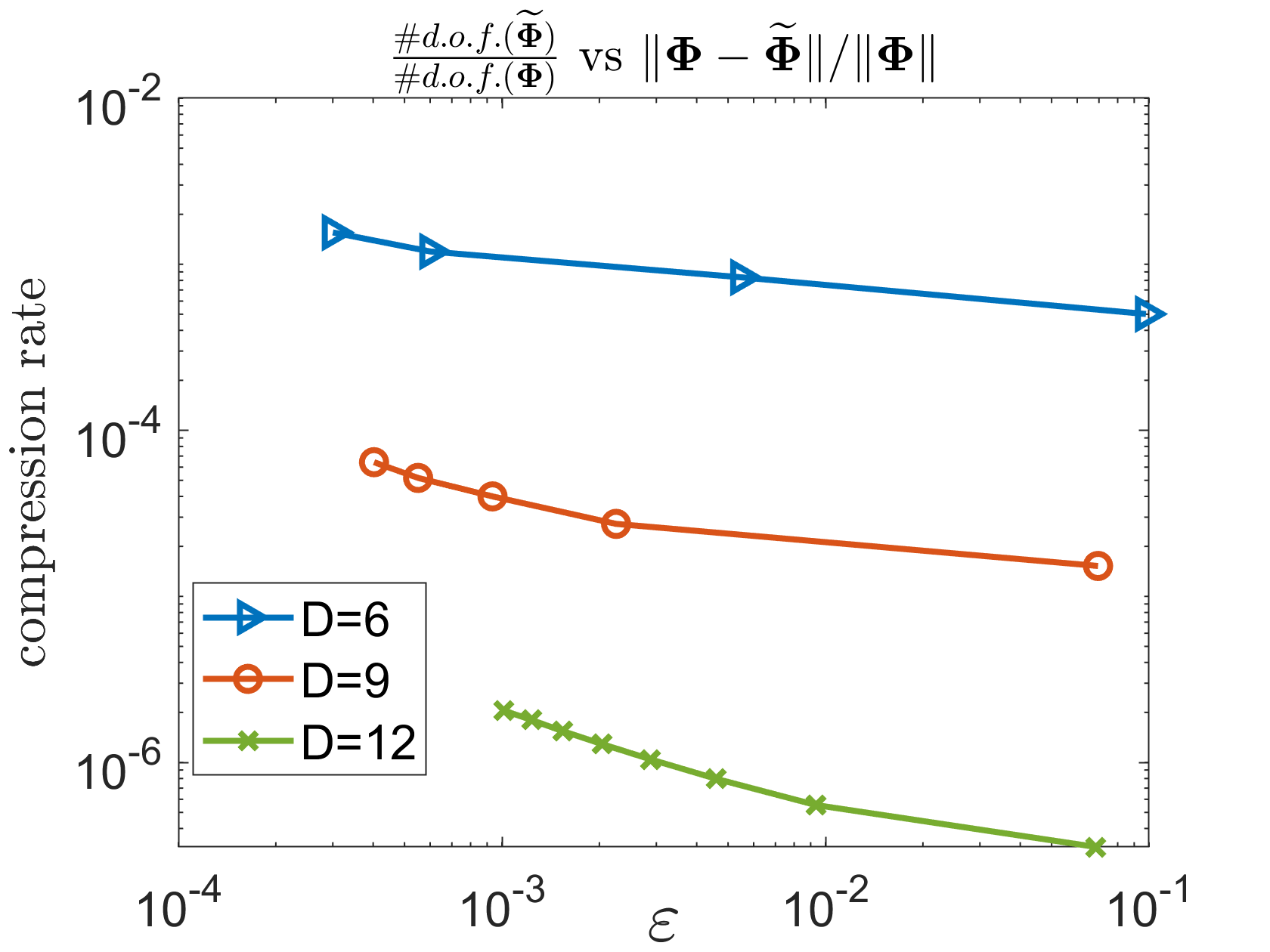} 
\includegraphics[width=0.32\textwidth]{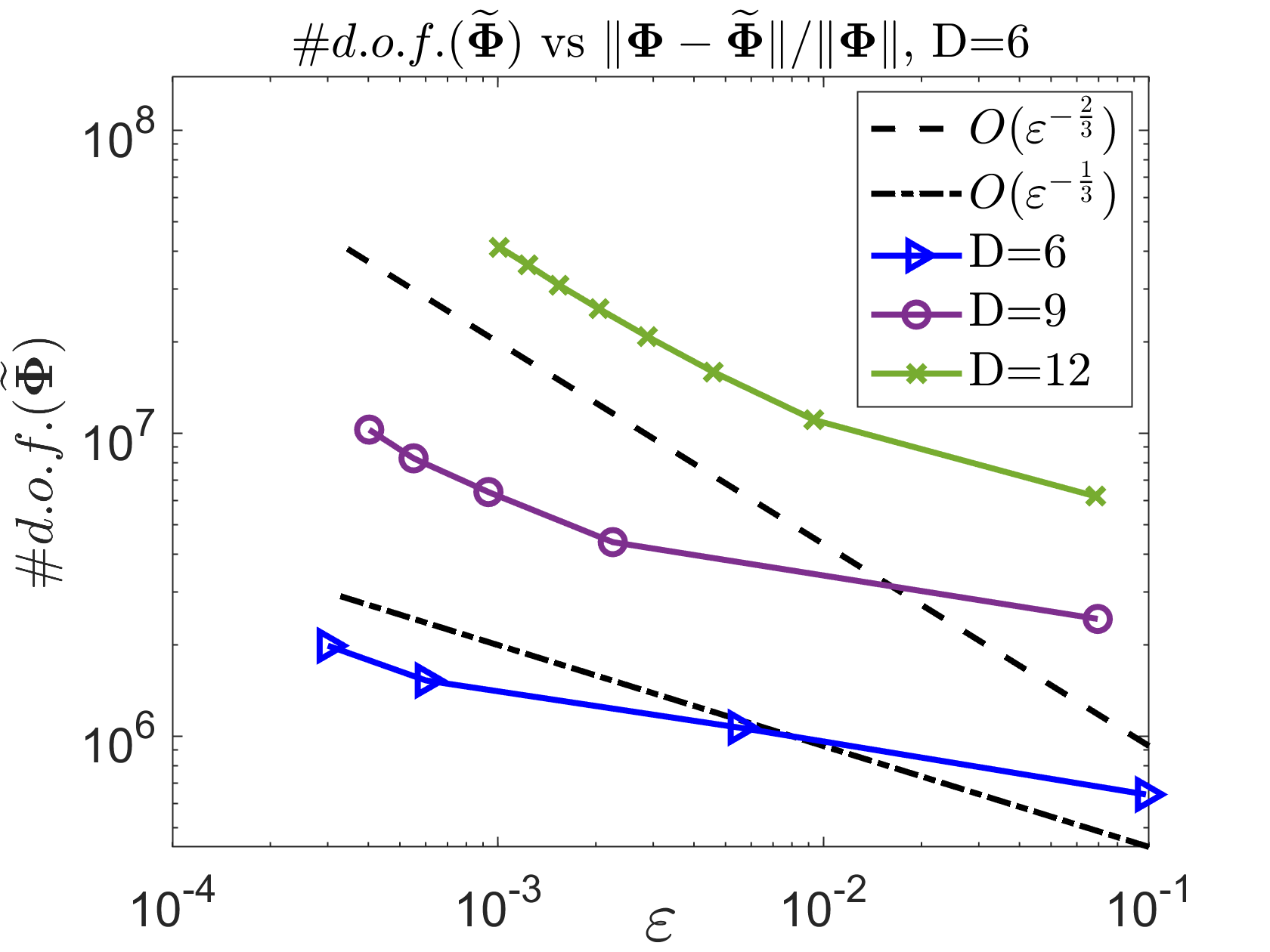} 
\end{center}
\caption{Completion statistics versus completion accuracy for parameter space dimensions $D=6,9,12$.
Left panel: Percent of the observed tensor elements; Central panel: Compression factor; Right panel: Total number of degrees of freedom for representation in HTT format.  }
\label{fig:Dall}
\end{figure}

\begin{figure}[ht]
\begin{center}
\includegraphics[width=0.32\textwidth]{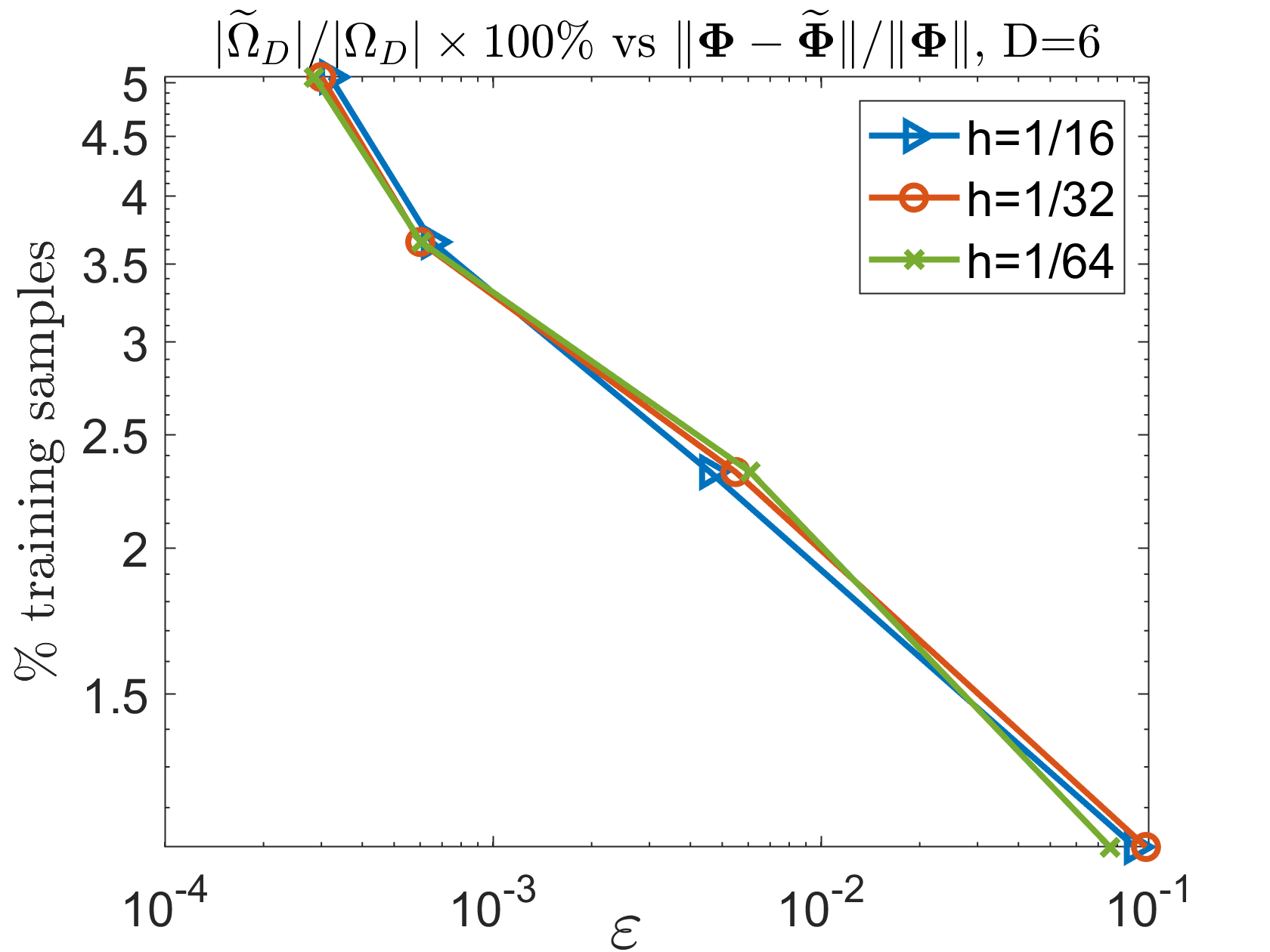} 
\includegraphics[width=0.33\textwidth]{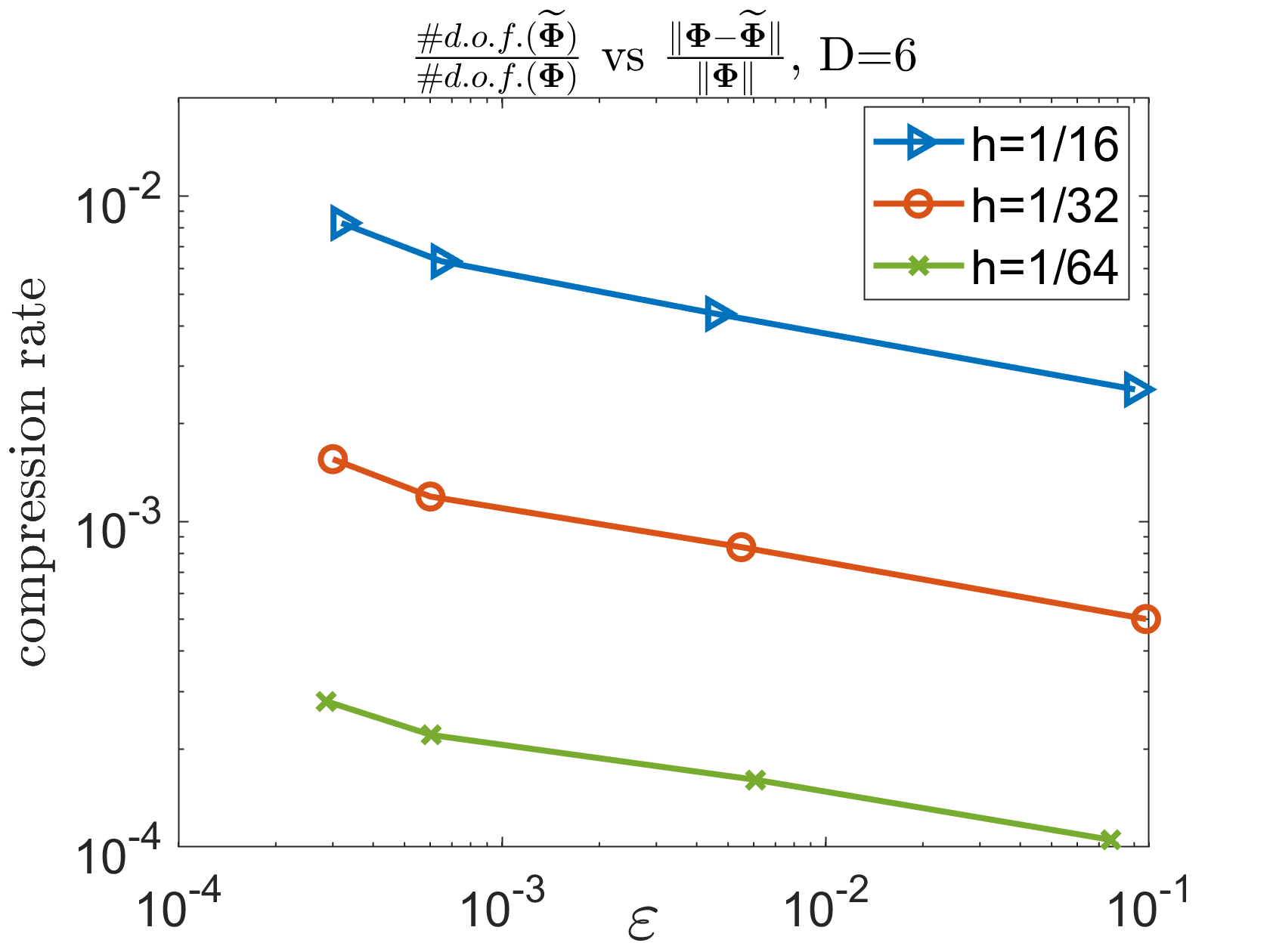} 
\includegraphics[width=0.32\textwidth]{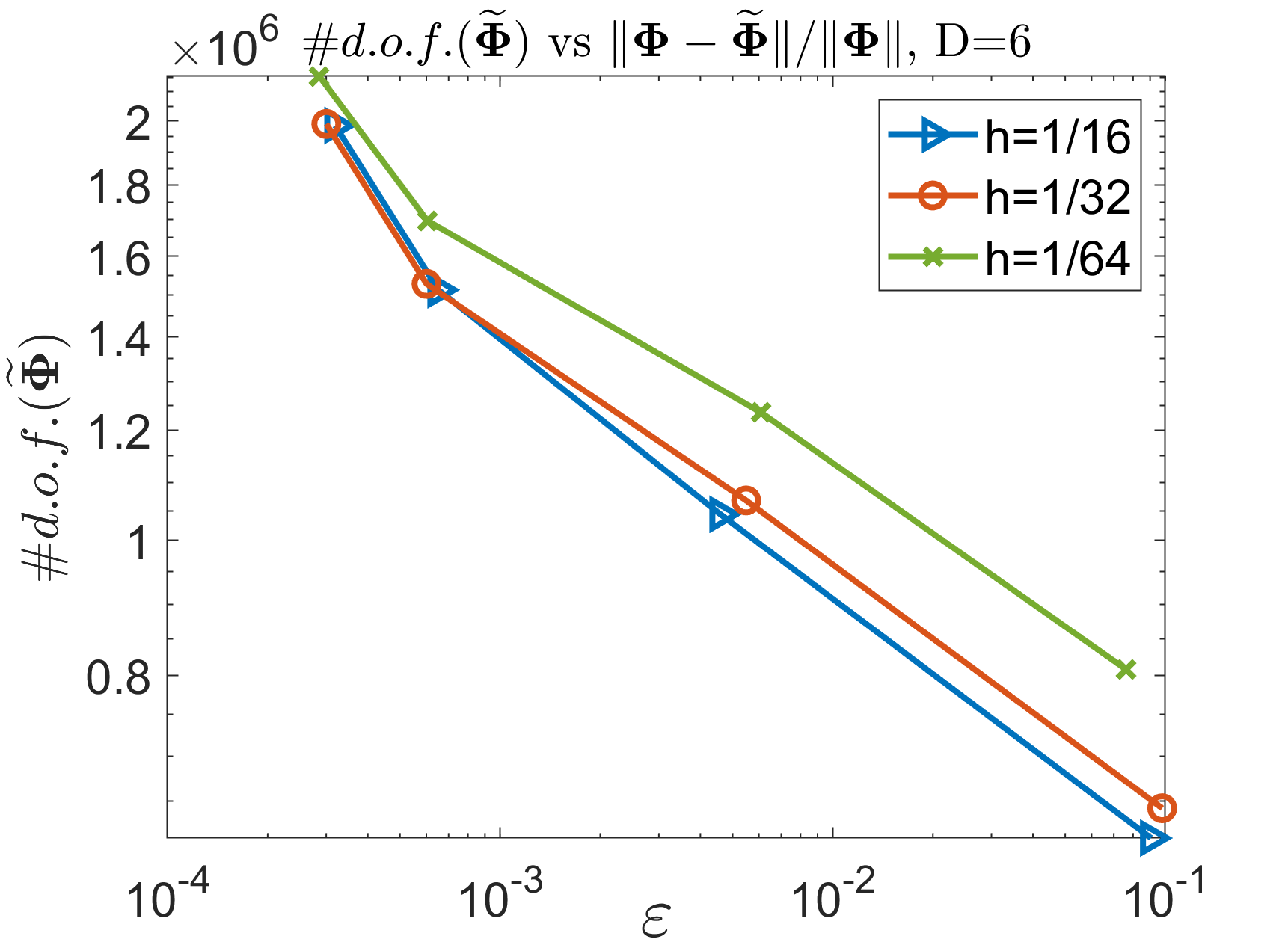} 
\end{center}
\caption{Completion statistics versus completion accuracy for various spatial resolutions.
Left panel: Percent of the observed tensor elements; Central panel: Compression factor; Right panel: Total number of degrees of freedom for representation in HTT format.}
\label{fig:D6}
\end{figure}

\begin{figure}[ht]
\begin{center}
\includegraphics[width=0.30\textwidth]{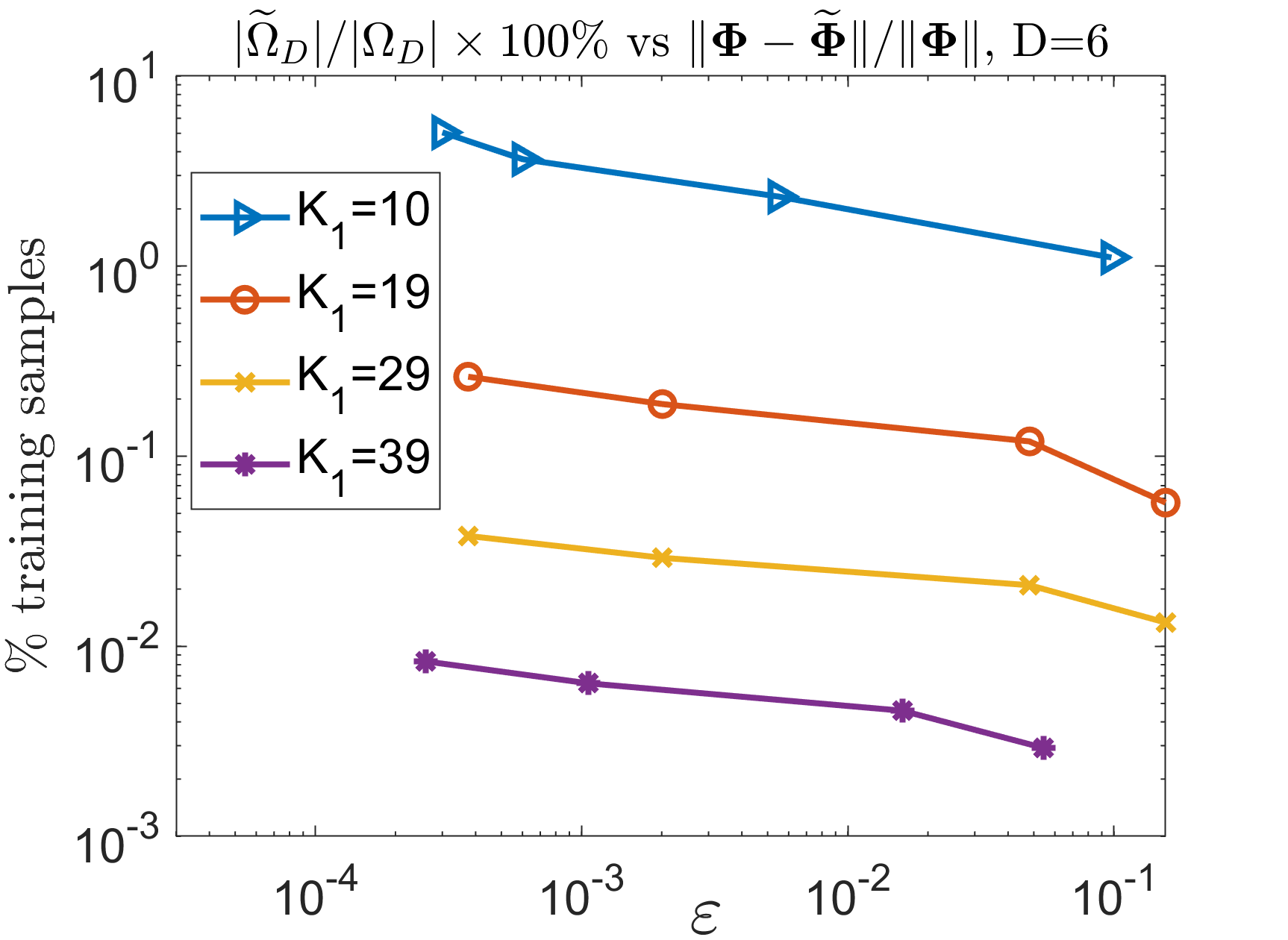} 
\includegraphics[width=0.34\textwidth]{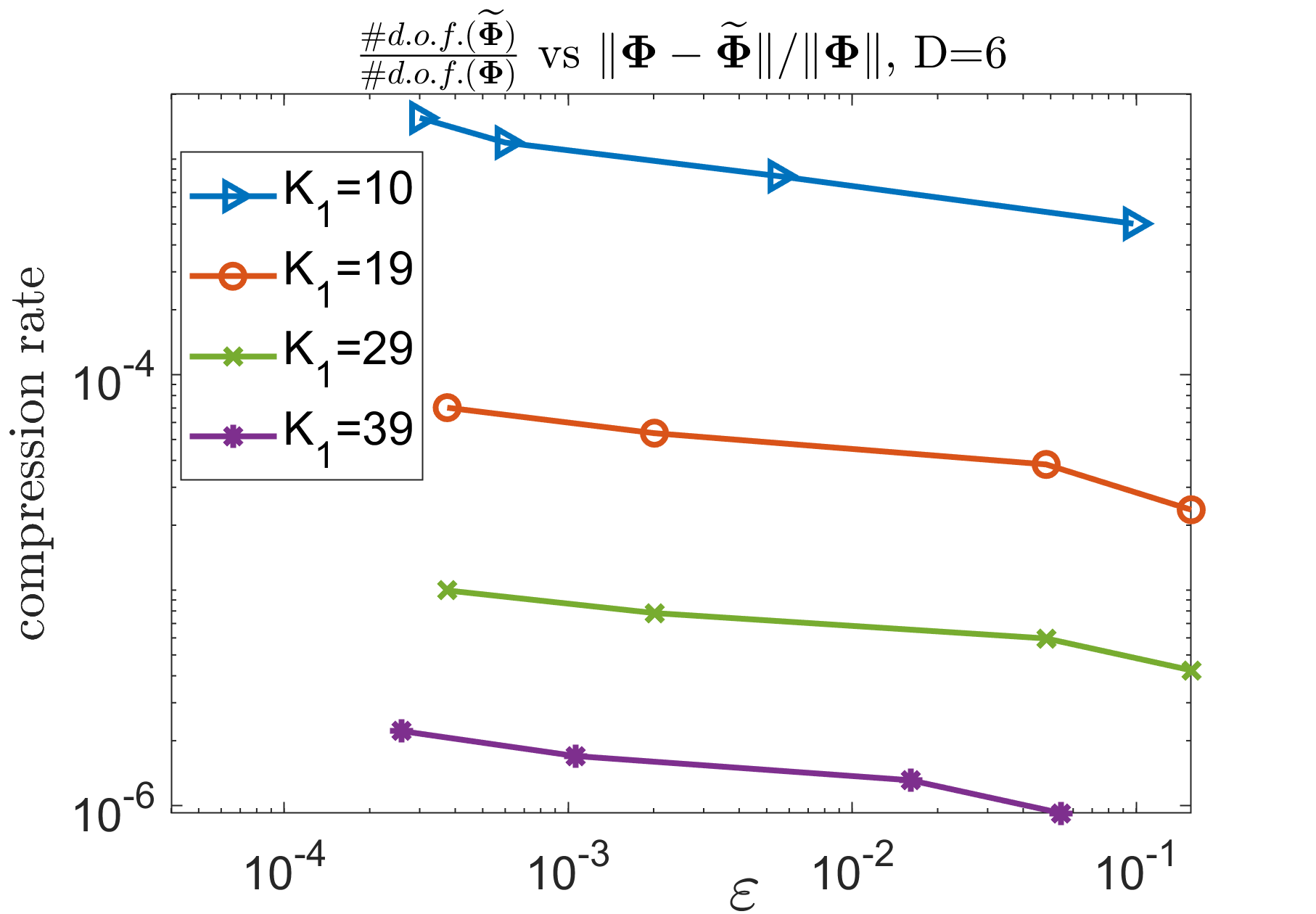} 
\includegraphics[width=0.34\textwidth]{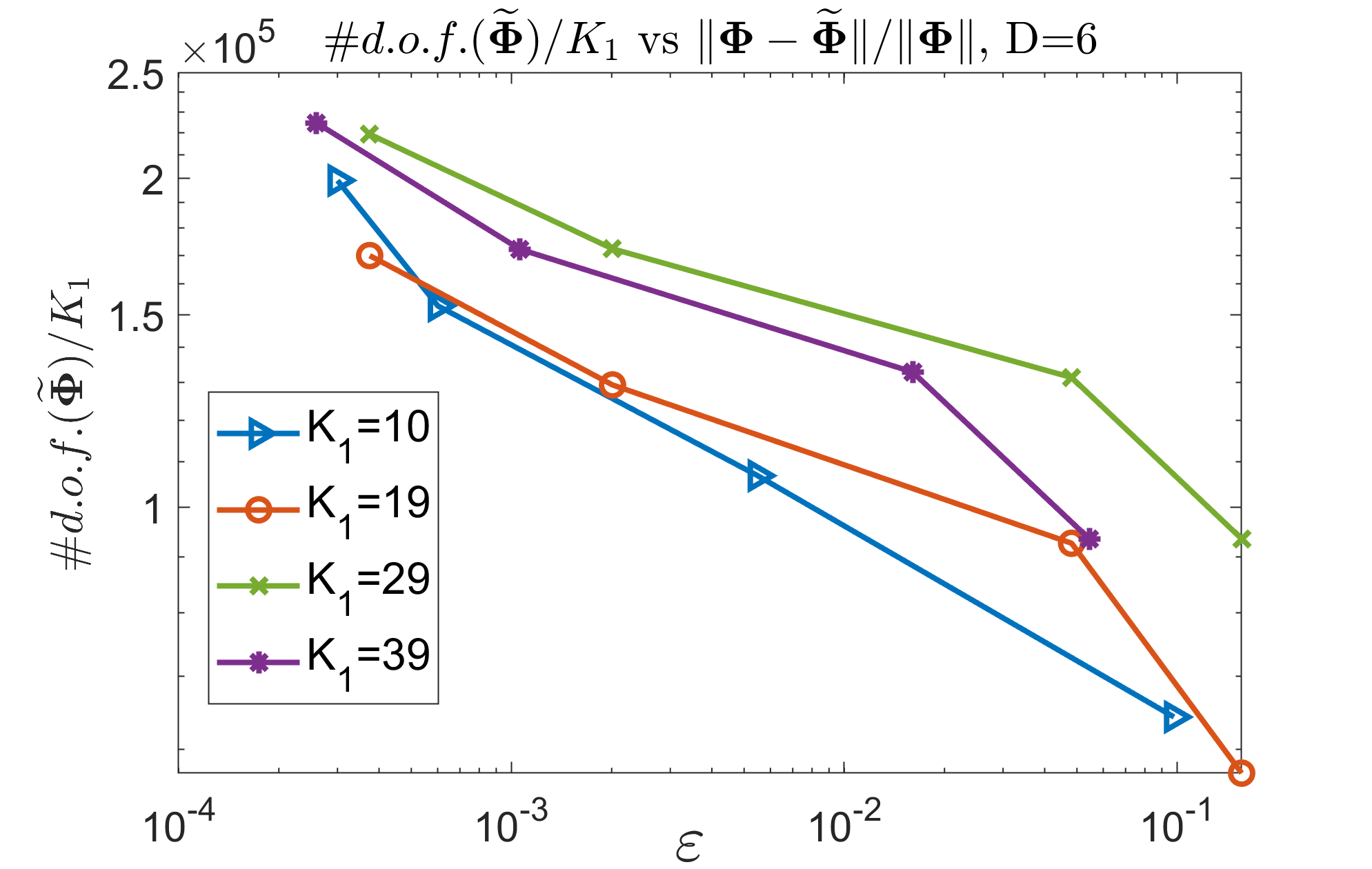} 
\end{center}
\caption{Completion statistics versus completion accuracy for various resolutions of the parameter domain.
Left panel: Percent of the observed tensor elements; Central panel: Compression factor as a ratio in HTT format; Right panel: A rescaled total number of degrees of freedom for representation in HTT format.}
\label{fig:D6b}
\end{figure}

The percentage of observed entries of the snapshot tensor $\boldsymbol{\Phi}$ required for its completion depends on $D$ and the targeted accuracy $\varepsilon$; see the left panel in Figure~\ref{fig:Dall}. It decreases from $5\%$ for $D = 6$ to $0.006\%$ for $D = 12$ as required for completion with $\varepsilon = 5 \times 10^{-4}$ accuracy; see the ratio $|\widetilde{\Omega}_D| / |\Omega_D|$ in Table~\ref{tab:advdiffD6a}.

Other completion statistics presented in Figure~\ref{fig:Dall} include the compression factor and the total degrees of freedom (d.o.f.) required for representation in the HTT format. The compression factor, defined as the ratio between the number of d.o.f. in HTT and full tensor format, and the total number of d.o.f., grows as $\varepsilon$ decreases, as expected, reaching around $10^6$ times compression for $D = 12$. For the total d.o.f. in HTT format, the observed rate of growth with $\varepsilon \to 0$ lies between $O(\varepsilon^{-\frac{1}{3}})$ and $O(\varepsilon^{-\frac{2}{3}})$; see Figure~\ref{fig:Dall}, right panel. 

We now study how the completion statistics change when the discretization parameters vary. 
Specifically, we vary the space discretization parameter $h$ and refine the mesh in the parameter domain, keeping the number of parameters fixed ($D = 6$). 
Table~\ref{tab:advdiffD6b} shows that both $C$ and $D$ ranks are not sensitive to the refinement or coarsening of meshes in both physical and parameter domains.  

The percentage of observed entries of the snapshot tensor $\boldsymbol{\Phi}$ required for its completion is essentially independent of the space resolution (left panel in Figure~\ref{fig:D6}) and decreases steeply as the mesh in the parameter domain is refined (left panel in Figure~\ref{fig:D6b}). The dependence of the compression factors and the total d.o.f. in HTT format on $\varepsilon$ remains consistent across refined meshes in both parameter and physical domains. 
In particular, the dependence of $\#\text{d.o.f.}(\widetilde{\boldsymbol{\Phi}})$ on $h$ is very mild and results only from the growth of ``external'' dimensions $M_i$ of the $\mathbf{U}$-matrices in \eqref{eq:HTT} (right panel in Figure~\ref{fig:D6}).
A similar observation holds for the dependence of $\#\text{d.o.f.}(\widetilde{\boldsymbol{\Phi}})$ on parameter mesh refinement (right panel in Figure~\ref{fig:D6b}).
Both are a consequence of the fact that $C$ and $D$ ranks are (almost) independent of these mesh refinements (see Table~\ref{tab:advdiffD6b}). 

 {The final experiment for the advection-diffusion equation is to compare the performance of CTROM to the conventional TROM that uses the full tensor of snapshots, as introduced in \cite{mamonov2022interpolatory}. For this comparison the setup is as follows. We use $D=9$ parameters and the grid $\hcA$ with $K_1 = 5$ and $K_j = 3$, $j=2,\ldots,D$, for a total of $32,805$ grid points. While $\hcA$ is rather coarse, it is still the finest grid that can be handled by the conventional TROM for $D=9$ on the machine used for the numerical experiments. We set $\eps = 5 \cdot 10^{-4}$ in Algorithm~\ref{Alg1a} and we use the same accuracy for TT decomposition in the offline stage of the conventional TROM. We compare the two TROM aproaches by timing the offline stage of both, since it dominates the computational cost. For the conventional TROM the computation of the snapshots for the tensor $\bPhi$ took $28$hrs $51$min with additional $14$min to perform its TT decomposition for a total of $29$hrs $05$min for the offline stage. The offline stage of CTROM, Algorithm~\ref{Alg1a}, took $7$hrs $10$min which includes computing $4,172$ samples, i.e., the percentage of observed entries of $\bPhi$ is $12.71\%$. Thus, for this particular case we observe more than four-fold gain in computational cost of CTROM compared to the conventional TROM. Note also that the relatively high observed entry percentage in Algorithm~\ref{Alg1a} is due to a coarse grid $\hcA$. One should expect even greater performance gains of CTROM for finer grids $\hcA$.}

We conclude this section by noting that the accuracy of the HTT-ROM is found to be insensitive to variations in dimensions and meshes (see the error statistics in Tables~\ref{tab:advdiffD6a} and~\ref{tab:advdiffD6b}). It is determined by the accuracy of the tensor completion and $\ell$.
We consistently use $\ell = 11$, which is sufficiently large for the HTT-ROM error to be dominated by the completion error.

\section{Conclusions}

Motivated by the problem of reconstructing the parametric solution manifold of a dynamical system through sparse sampling, this paper introduces a low-rank tensor format for the inexact completion of a tensor observed slice-wise. The completion procedure is based on solving a sequence of ``standard'' tensor completion problems using a common format, with the tensor train format selected in this work. These standard completion problems can be addressed in parallel.

Applying the completion method to two linear parametric parabolic PDEs, discretized via the finite element method, revealed the following properties: (i) The resulting tensor ranks are insensitive to discretization parameters in space and time, as well as to grid resolution in the parameter domain. (ii) The ranks increase with more accurate completion, requiring additional memory for storing the recovered tensor. This dependency is of the form 
$
\#\text{d.o.f.} = O(\eps^{-\alpha})
$,
where $\eps$ represents the completion accuracy and $\alpha \ge 0$ is an exponent. In our tests, $\alpha$ did not exceed $\frac{2}{3}$. (iii) The total number and percentage of observed tensor entries required for successful completion depends on the target accuracy, but more significantly on the dimension of the parameter space and on how fine is the parameter domain grid. Specifically, the total number of entries required increases roughly linearly with the number of nodes for any parameter. Therefore, if the grid in $\mathcal{A}$ is refined by doubling the nodes in each direction, the number of observed entries required doubles, while the percentage of observed entries `decrease' by a factor of $2^{D-1}$. This explains our ability to recover snapshot tensors from fewer than $0.01\%$ of entries in higher-dimensional parameter spaces. (iv) The compression achieved by the HTT format for snapshot tensors similarly depends on parameter space dimension and resolution. In our numerical examples, we achieved a compression factor of $10^6$ while preserving an accuracy of $10^{-4}$.

In the context of parametric dynamical systems, we found the available rank-adaptive tensor completion method effective and efficient for problems with up to a dozen parameters. Extending this approach significantly beyond this number of parameters warrants further investigation.

\begin{acknowledgements}
 A.M. and M.O. were supported by the U.S. National Science Foundation under award DMS-2309197.
This material is based upon research supported in part by the U.S. Office of Naval Research 
under award number N00014-21-1-2370 to A.M. The authors are grateful to the anonymous reviewers for their insightful comments and suggestions, which led to several improvements in the manuscript.
\end{acknowledgements}

\bibliographystyle{spmpsci}      
\bibliography{literatur}{}

@article{mamonov2022interpolatory,
	title={Interpolatory tensorial reduced order models for parametric dynamical systems},
	author={Mamonov, Alexander V and Olshanskii, Maxim A},
	journal={Computer Methods in Applied Mechanics and Engineering},
	volume={397},
	pages={115122},
	year={2022},
	publisher={Elsevier}
}

@article{mamonov2024tensorial,
  title={Tensorial parametric model order reduction of nonlinear dynamical systems},
  author={Mamonov, Alexander V and Olshanskii, Maxim A},
  journal={SIAM Journal on Scientific Computing},
  volume={46},
  number={3},
  pages={A1850--A1878},
  year={2024},
  publisher={SIAM}
}

@article{mamonov2023analysis,
  title={A priori analysis of a tensor {ROM} for parameter dependent parabolic problems},
  author={Mamonov, Alexander V and Olshanskii, Maxim A},
  journal={SIAM Journal on Numerical Analysis},
  volume={63},
  number={1},
  pages={239--261},
  year={2025},
  publisher={SIAM}
}

@article{candes2010power,
  title={The power of convex relaxation: {N}ear-optimal matrix completion},
  author={Cand{\`e}s, Emmanuel J and Tao, Terence},
  journal={IEEE transactions on information theory},
  volume={56},
  number={5},
  pages={2053--2080},
  year={2010},
  publisher={IEEE}
}

@article{candes2012exact,
  title={Exact matrix completion via convex optimization},
  author={Candes, Emmanuel and Recht, Benjamin},
  journal={Communications of the ACM},
  volume={55},
  number={6},
  pages={111--119},
  year={2012},
  publisher={ACM New York, NY, USA}
}

@article{cai2010singular,
  title={A singular value thresholding algorithm for matrix completion},
  author={Cai, Jian-Feng and Cand{\`e}s, Emmanuel J and Shen, Zuowei},
  journal={SIAM Journal on optimization},
  volume={20},
  number={4},
  pages={1956--1982},
  year={2010},
  publisher={SIAM}
}

@article{nguyen2019low,
  title={Low-rank matrix completion: A contemporary survey},
  author={Nguyen, Luong Trung and Kim, Junhan and Shim, Byonghyo},
  journal={IEEE Access},
  volume={7},
  pages={94215--94237},
  year={2019},
  publisher={IEEE}
}

@inproceedings{jain2013low,
  title={Low-rank matrix completion using alternating minimization},
  author={Jain, Prateek and Netrapalli, Praneeth and Sanghavi, Sujay},
  booktitle={Proceedings of the forty-fifth annual ACM symposium on Theory of computing},
  pages={665--674},
  year={2013}
}

@article{vandereycken2013low,
  title={Low-rank matrix completion by {Riemannian} optimization},
  author={Vandereycken, Bart},
  journal={SIAM Journal on Optimization},
  volume={23},
  number={2},
  pages={1214--1236},
  year={2013},
  publisher={SIAM}
}

@article{steinlechner2016riemannian,
  title={Riemannian optimization for high-dimensional tensor completion},
  author={Steinlechner, Michael},
  journal={SIAM Journal on Scientific Computing},
  volume={38},
  number={5},
  pages={S461--S484},
  year={2016},
  publisher={SIAM}
}

@article{garreis2017constrained,
  title={Constrained optimization with low-rank tensors and applications to parametric problems with {PDEs}},
  author={Garreis, Sebastian and Ulbrich, Michael},
  journal={SIAM Journal on Scientific Computing},
  volume={39},
  number={1},
  pages={A25--A54},
  year={2017},
  publisher={SIAM}
}

@article{candes2010matrix,
  title={Matrix completion with noise},
  author={Candes, Emmanuel J and Plan, Yaniv},
  journal={Proceedings of the IEEE},
  volume={98},
  number={6},
  pages={925--936},
  year={2010},
  publisher={IEEE}
}

@article{de2000multilinear,
  title={A multilinear singular value decomposition},
  author={De Lathauwer, Lieven and De Moor, Bart and Vandewalle, Joos},
  journal={SIAM journal on Matrix Analysis and Applications},
  volume={21},
  number={4},
  pages={1253--1278},
  year={2000},
  publisher={SIAM}
}

@article{dolgov2018direct,
  title={Direct tensor-product solution of one-dimensional elliptic equations with parameter-dependent coefficients},
  author={Dolgov, Sergey V and Kazeev, Vladimir A and Khoromskij, Boris N},
  journal={Mathematics and computers in simulation},
  volume={145},
  pages={136--155},
  year={2018},
  publisher={Elsevier}
}

@article{kressner2014low,
  title={Low-rank tensor completion by {Riemannian} optimization},
  author={Kressner, Daniel and Steinlechner, Michael and Vandereycken, Bart},
  journal={BIT Numerical Mathematics},
  volume={54},
  pages={447--468},
  year={2014},
  publisher={Springer}
}

@article{bengua2017efficient,
  title={Efficient tensor completion for color image and video recovery: Low-rank tensor train},
  author={Bengua, Johann A and Phien, Ho N and Tuan, Hoang Duong and Do, Minh N},
  journal={IEEE Transactions on Image Processing},
  volume={26},
  number={5},
  pages={2466--2479},
  year={2017},
  publisher={IEEE}
}

@article{grasedyck2019stable,
  title={Stable {ALS} approximation in the {TT}-format for rank-adaptive tensor completion},
  author={Grasedyck, Lars and Kr{\"a}mer, Sebastian},
  journal={Numerische Mathematik},
  volume={143},
  number={4},
  pages={855--904},
  year={2019},
  publisher={Springer}
}

@article{ballani2016reduced,
  title={Reduced basis methods: from low-rank matrices to low-rank tensors},
  author={Ballani, Jonas and Kressner, Daniel},
  journal={SIAM Journal on Scientific Computing},
  volume={38},
  number={4},
  pages={A2045--A2067},
  year={2016},
  publisher={SIAM}
}

@article{signoretto2010nuclear,
  title={Nuclear norms for tensors and their use for convex multilinear estimation},
  author={Signoretto, Marco and De Lathauwer, Lieven and Suykens, Johan AK},
  journal={Submitted to Linear Algebra and Its Applications},
  volume={43},
  year={2010}
}

@article{gandy2011tensor,
  title={Tensor completion and low-n-rank tensor recovery via convex optimization},
  author={Gandy, Silvia and Recht, Benjamin and Yamada, Isao},
  journal={Inverse problems},
  volume={27},
  number={2},
  pages={025010},
  year={2011},
  publisher={IOP Publishing}
}

@article{liu2022efficient,
  title={Efficient tensor completion methods for 5-{D} seismic data reconstruction: Low-rank tensor train and tensor ring},
  author={Liu, Dawei and Sacchi, Mauricio D and Chen, Wenchao},
  journal={IEEE Transactions on Geoscience and Remote Sensing},
  volume={60},
  pages={1--17},
  year={2022},
  publisher={IEEE}
}

@article{yokota2016smooth,
  title={Smooth {PARAFAC} decomposition for tensor completion},
  author={Yokota, Tatsuya and Zhao, Qibin and Cichocki, Andrzej},
  journal={IEEE Transactions on Signal Processing},
  volume={64},
  number={20},
  pages={5423--5436},
  year={2016},
  publisher={IEEE}
}

@article{zhao2015bayesian,
  title={Bayesian {CP} factorization of incomplete tensors with automatic rank determination},
  author={Zhao, Qibin and Zhang, Liqing and Cichocki, Andrzej},
  journal={IEEE transactions on pattern analysis and machine intelligence},
  volume={37},
  number={9},
  pages={1751--1763},
  year={2015},
  publisher={IEEE}
}

@article{zhao2015bayesian2,
  title={Bayesian sparse tucker models for dimension reduction and tensor completion},
  author={Zhao, Qibin and Zhang, Liqing and Cichocki, Andrzej},
  journal={arXiv preprint arXiv:1505.02343},
  year={2015}
}

@inproceedings{rauhut2015tensor,
  title={Tensor completion in hierarchical tensor representations},
  author={Rauhut, Holger and Schneider, Reinhold and Stojanac, {\v{Z}}eljka},
  booktitle={Compressed Sensing and its Applications: MATHEON Workshop 2013},
  pages={419--450},
  year={2015},
  organization={Springer}
}

@article{eigel2017adaptive,
  title={Adaptive stochastic {Galerkin FEM} with hierarchical tensor representations},
  author={Eigel, Martin and Pfeffer, Max and Schneider, Reinhold},
  journal={Numerische Mathematik},
  volume={136},
  pages={765--803},
  year={2017},
  publisher={Springer}
}

@article{bachmayr2017kolmogorov,
  title={Kolmogorov widths and low-rank approximations of parametric elliptic {PDEs}},
  author={Bachmayr, Markus and Cohen, Albert},
  journal={Mathematics of Computation},
  volume={86},
  number={304},
  pages={701--724},
  year={2017}
}

@article{bachmayr2018parametric,
  title={Parametric {PDEs}: sparse or low-rank approximations?},
  author={Bachmayr, Markus and Cohen, Albert and Dahmen, Wolfgang},
  journal={IMA Journal of Numerical Analysis},
  volume={38},
  number={4},
  pages={1661--1708},
  year={2018},
  publisher={Oxford University Press}
}

@article{bachmayr2023low,
  title={Low-rank tensor methods for partial differential equations},
  author={Bachmayr, Markus},
  journal={Acta Numerica},
  volume={32},
  pages={1--121},
  year={2023},
  publisher={Cambridge University Press}
}

@article{olshanskii2024approximating,
  title={Approximating a branch of solutions to the {Navier--Stokes} equations by reduced-order modeling},
  author={Olshanskii, Maxim A and Rebholz, Leo G},
  journal={Journal of Computational Physics},
  volume={524},
  pages={113728},
  year={2025},
  publisher={Elsevier}
}

@book{hackbusch2012tensor,
	title={Tensor spaces and numerical tensor calculus},
	author={Hackbusch, Wolfgang},
	volume={42},
	year={2012},
	publisher={Springer}
}

@article{kressner2011low,
	title={Low-rank tensor {Krylov} subspace methods for parametrized linear systems},
	author={Kressner, Daniel and Tobler, Christine},
	journal={SIAM Journal on Matrix Analysis and Applications},
	volume={32},
	number={4},
	pages={1288--1316},
	year={2011},
	publisher={SIAM}
}

@article{khoromskij2011tensor,
title={Tensor-structured Galerkin approximation of parametric and stochastic elliptic {PDEs}},
author={Khoromskij, Boris N and Schwab, Christoph},
journal={SIAM Journal on Scientific Computing},
volume={33},
number={1},
pages={364--385},
year={2011},
publisher={SIAM}
}

@article{nouy2017low,
	title={Low-rank methods for high-dimensional approximation and model order reduction},
	author={Nouy, Anthony},
	journal={Model reduction and approximation, P. Benner, A. Cohen, M. Ohlberger, and K. Willcox, eds., SIAM, Philadelphia, PA},
	pages={171--226},
	year={2017}
}

@article{schwab2011sparse,
  title={Sparse tensor discretizations of high-dimensional parametric and stochastic {PDEs}},
  author={Schwab, Christoph and Gittelson, Claude Jeffrey},
  journal={Acta Numerica},
  volume={20},
  pages={291--467},
  year={2011},
  publisher={Cambridge University Press}
}

@article{ballani2015hierarchical,
  title={Hierarchical tensor approximation of output quantities of parameter-dependent {PDEs}},
  author={Ballani, Jonas and Grasedyck, Lars},
  journal={SIAM/ASA Journal on Uncertainty Quantification},
  volume={3},
  number={1},
  pages={852--872},
  year={2015},
  publisher={SIAM}
}

@article{dolgov2015polynomial,
  title={Polynomial chaos expansion of random coefficients and the solution of stochastic partial differential equations in the tensor train format},
  author={Dolgov, Sergey and Khoromskij, Boris N and Litvinenko, Alexander and Matthies, Hermann G},
  journal={SIAM/ASA Journal on Uncertainty Quantification},
  volume={3},
  number={1},
  pages={1109--1135},
  year={2015},
  publisher={SIAM}
}

@article{glau2020low,
  title={Low-rank tensor approximation for Chebyshev interpolation in parametric option pricing},
  author={Glau, Kathrin and Kressner, Daniel and Statti, Francesco},
  journal={SIAM Journal on Financial Mathematics},
  volume={11},
  number={3},
  pages={897--927},
  year={2020},
  publisher={SIAM}
}

@article{dolgov2019hybrid,
title={A Hybrid Alternating Least Squares--{TT}-Cross Algorithm for Parametric {PDE}s},
author={Dolgov, Sergey and Scheichl, Robert},
journal={SIAM/ASA Journal on Uncertainty Quantification},
volume={7},
number={1},
pages={260--291},
year={2019},
publisher={SIAM}
}

\end{document}